\documentclass[11pt, a4paper]{amsart}
\usepackage[a4paper, margin=3cm]{geometry}
\usepackage{amssymb,amsfonts}
\usepackage{enumitem}

\usepackage{latexsym}
\usepackage{color}

\newtheorem{Theorem}{Theorem}[section]

\newtheorem{Lemma}[Theorem]{Lemma}
\newtheorem{Corollary}[Theorem]{Corollary}

\newtheorem{Proposition}[Theorem]{Proposition}

\theoremstyle{definition}
\newtheorem{Definition}[Theorem]{Definition}
\newtheorem{Remark}[Theorem]{Remark}
\newtheorem{Example}[Theorem]{Example}

\numberwithin{equation}{section}

\def\C{\mathbb C}

\def\N{\mathbb N}

\def\R{\mathbb R}

\def\Z{\mathbb Z}

\newcommand{\CA}{\mathcal{A}}
\newcommand{\CB}{\mathcal{B}}

\newcommand{\CE}{\mathcal{E}}
\newcommand{\CF}{\mathcal{F}}

\newcommand{\CH}{\mathcal{H}}
\newcommand{\CI}{\mathcal{I}}

\newcommand{\CK}{\mathcal{K}}

\newcommand{\CU}{\mathcal{U}}

\newcommand{\FA}{\mathfrak{A}}
\newcommand{\FH}{\mathfrak{H}}
\newcommand{\FN}{\mathfrak{N}}

\def\be{\begin{equation}}
\def\ee{\end{equation}}
\def\bt{\begin{Theorem}}
\def\et{\end{Theorem}}
\def\bi{\begin{itemize}}
\def\ei{\end{itemize}}
\def\bea{\begin{eqnarray}}
\def\eea{\end{eqnarray}}
\def\beast{\begin{eqnarray*}}
\def\eeast{\end{eqnarray*}}
\def\ben{\begin{enumerate}}
\def\een{\end{enumerate}}

\def\bi{\bibitem}

\def\rar{\rightarrow}

\newcommand{\fns}{f.n.s.\ }
\newcommand{\half}{{\frac{1}{2}}}

\DeclareMathOperator{\re}{Re}

\DeclareMathOperator{\ind}{Ind}
\DeclareMathOperator{\ccr}{CCR}
\DeclareMathOperator{\Tr}{Tr}
\DeclareMathOperator{\Ad}{Ad}
\DeclareMathOperator{\cspan}{\overline{span}}
\newcommand{\walpha}{\widetilde{\alpha}}

\newcommand{\wgamma}{\widetilde{\gamma}}

\renewcommand{\MR}[1]{} 

\begin{document}
\title[On the classification and modular extendability of E$_0$-semigroups]{On the classification and modular extendability of E$_0$-semigroups on factors}

\author{Panchugopal Bikram}
\author{Daniel Markiewicz}
\date{October 24, 2014}
\address{Department of Mathematics,
Ben-Gurion University of the Negev,
P.O.B. 653, Be'er Sheva 8410501,
Israel.}
\email{bikram@math.bgu.ac.il, danielm@math.bgu.ac.il}


\thanks{The first author was supported in part by a postdoctoral fellowship funded in part by the 
Skirball Foundation via the Center for Advanced Studies in Mathematics at Ben-Gurion University of 
the Negev. The second author was supported by the ISF within the ISF-UGC‬‬ 
‫‪joint research program framework‬‬ (grant No. 1775/14).}

\begin{abstract}
In this paper we study modular extendability and equimodularity of 
endomorphisms and E$_0$-semigroups on factors with respect to \fns weights. We show that modular extendability is a property that does not depend on the choice of weights, it is a cocycle conjugacy invariant and it is preserved under tensoring.
We say that a modularly extendable E$_0$-semigroup is of type EI, EII or EIII
if its modular extension is of type I, II or III, respectively.
We prove that all types exist on properly infinite factors. 

We also compute the coupling index and the relative commutant index for the CAR flows and $q$-CCR flows.
As an application, by considering repeated tensors of the CAR flows we show that there are infinitely many non cocycle 
conjugate non-extendable  $E_0$-semigroups on the hyperfinite factors of
types II$_1$, II$_{\infty}$ and III$_\lambda$, for $\lambda \in (0,1)$. 
\end{abstract} 
\maketitle 

\section{Introduction}
A weak$^*$ continuous one-parameter semigroup of unital $*$-endomorphisms on a von Neumann algebra is called an $E_0$-semigroup,
and there has been considerable interest in their classification up to the
equivalence relation called cocycle conjugacy. Most of the progress
has focused on the case of E$_0$-semigroups on type I$_\infty$ factors:
those are divided into types I, II and III, and every such
E$_0$-semigroup gives rise to a product system of Hilbert spaces. 
In fact, the classification theory of E$_0$-semigroups of type
I$_\infty$  factors is equivalent to the classification problem
of product systems of Hilbert spaces up to isomorphism.
For an overview of the theory of E$_0$-semigroups on type I$_\infty$ factors, we recommend
the monograph by Arveson~\cite{Arv}.

E$_0$-semigroups on II$_1$ factors were first studied by Powers~\cite{powers}, who introduced an index for their study. 
Alevras~\cite{Alev} computed the Powers index of the Clifford flows on type II$_1$ factors, however the index is not known
to be a cocycle conjugacy invariant. On the other hand, Alevras~\cite{alevras-thesis,Alev} also showed  that a product system
of W*-correspondences can be associated to every E$_0$-semigroup
on a type II$_1$ factor, and the isomorphism class of the product system is a cocycle conjugacy invariant. In fact, the association of product
systems of W*-correspondences to E$_0$-semigroups on general von Neumann algebras has been established by Bhat and 
Skeide~\cite{bhat-skeide-2000} and subsequent work of Skeide (see \cite{skeide-holyoke}). However product systems are difficult to compute in practice. Amosov, Bulinskii and Shirokov~\cite{ABS} were the first
to examine the issue of extendability of E$_0$-semigroups on general 
factors. Bikram, Izumi, Srinivasan and Sunder~\cite{BISS} introduced
the concept of equimodularity for endomorphisms, and applied it to obtain convenient criteria for the existence of extensions. As an application, it was proved in \cite{BISS} that the CAR flows are not extendable on the hyperfinite factor of type II$_1$. Similarly, Bikram~\cite{Bikram} showed that the CAR flows are not extendable on hyperfinite III$_\lambda$ factors, for $\lambda \in (0,1)$, for a certain class of quasi-free states. In \cite{MS}, Srinivasan and Margetts introduced new invariants for E$_0$-semigroups on
type II$_1$ factors, especially the coupling index, and as an application showed that the Clifford flows are non-cocycle conjugate. Subsequently Margetts and Srinivasan~\cite{MS1} considered more general factors, and they showed that by varying the
quasi-free states appropriately, the CCR flows in hyperfinite type III$_\lambda$ factors are non-cocycle conjugate, for for $\lambda \in (0,1]$. They also proved that there are uncountably many non-cocycle conjugate E$_0$-semigroups on all hyperfinite II$_\infty$ and III$_\lambda$ factors, for $\lambda \in (0,1]$. 

In this paper we focus on the role of extendability in the classification of E$_0$-semigroups on factors, especially of types II and III. In Section~\ref{sec:extendibility}, we generalize the concept in \cite{BISS} of
extendability of a unital endomorphism  to the context of faithful normal semifinite \emph{weights}. Furthermore, we
show that this property, now renamed \emph{modular extendability} of
a unital endomorphism (or E$_0$-semigroup), does not depend on the choice of \fns weights. In addition, we show that the modular extension is a
cocycle conjugacy invariant of a modularly extendable E$_0$-semigroup. This allows us to
introduce in Section~\ref{sec:E-semigroups} a classification scheme based on the well-known situation for the type I$_\infty$ factor. Namely, we say that a modularly extendable E$_0$-semigroup has type EI, EII or EIII if its modular extension has 
type I, II or III, respectively. We prove that this is consistent with the classification of type I$_\infty$ factors. Furthermore, we show that 
the tensor product of two $E_0$-semigroups $\alpha$ and $\beta$ on factors are modularly extendable if and only if $\alpha \otimes \beta$
is modularly extendable. As an application, we show that if $M$ is a properly infinite factor, then there exist modularly extendable E$_0$-semigroups of every type on $M$.

In Section~\ref{sec:equimodularity}, we generalize the concept of equimodularity for \fns weights. We prove that the necessary condition for equimodularity described in \cite{BISS} for the case of states  in fact is necessary and sufficient (even in the context of
\fns weights). We also prove that the sufficient condition for modular extendability in the presence of equimodularity found in \cite{BISS} also holds for \fns weights. Despite the usefulness of these results, we show that for every properly infinite factor there exist modularly extendable unital endomorphisms and E$_0$-semigroups which are not
equimodular with respect to any \fns weight.

In Section~\ref{sec:E-semigroups}, we discuss the classification of E$_0$-semigroups on factors into types EI, EII and EIII and not extendable. We also consider two invariants for the classification of E$_0$-semigroups: the relative commutant index and coupling index. The relative commutant index was considered for type II$_1$ factors by Powers~\cite{powers} and Alevras~\cite{Alev}. The coupling index was introduced by Margetts and Srinivasan~\cite{MS, MS1}. 

Finally, in Section~\ref{sec:examples}, we apply our results to some concrete examples. 
Firstly, we show that CAR flows are equimodular but not modularly extendable for a class of quasi-free states even larger than that in \cite{BISS, Bikram};
we also compute their coupling and relative commutant indices. As an application, in analogy with Alevras's result for Clifford flows in \cite{Alev}, 
we show that under mild conditions, if $\alpha$ is a CAR flow then $\alpha^{\otimes^k}$ and $\alpha^{\otimes^\ell}$ are cocycle conjugate if and only if $k=\ell$. 
As another example, we consider the class of $q$-CCR flows on the $q$-Gaussian factors and prove  that they are equimodular but not modularly extendable, and compute their indices.

\emph{In this paper, inner products are linear in the first coordinate unless stated otherwise. We assume all Hilbert spaces to be separable and all von Neumann algebras have separable predual. We also assume
that our endomorphisms are *-preserving.}

\section{Extendability of endomorphisms}\label{sec:extendibility}

In this section we study extendibility of endomorphisms on factors.
This program started with the work of Amosov, Bulinskii and Shirovok~\cite{ABS}, and then Bikram, Izumi, Srinivasan and Sunder~\cite{BISS}. Since we extend some
of their results in a slightly different context, we review the situation
and introduce appropriate terminology.

 Let $M$ be a von Neumann algebra and let $\phi$ be a faithful normal semi-finite weight on M (in the continuation we will often use the abbreviation \emph{\fns weight}). 
 Such a pair $(M, \phi)$ will be called a \emph{non-commutative measure
 space.}
 Recall that $\phi$ has an associated GNS representation. Let
 $\CH_{\phi}$ be the quotient and completion of 
 $\FN_{\phi} = \{ x \in M ;~ \phi(x^*x) < +\infty \}$,
 and let  $\FN_{\phi} \ni x \mapsto x_{\phi} \in \CH_{\phi}$ denote 
 the canonical map.  The GNS $*$-representation $\pi_{\phi}: M \rar \CB(\CH_{\phi})$ is uniquely determined by the identity
 \[
  \langle \pi_{\phi}(a) x_{\phi}, y_{\phi} \rangle = \phi(y^*ax),
  \qquad a \in M,~ x, y \in \FN_{\phi}.
 \]
We denote by 
$J_{\phi}, \Delta_{\phi} \text{ and } \{\sigma_t^{\phi}\}$ the 
modular conjugation operator, modular operator and 
modular automorphism group, respectively,  for $M$ associated to $\phi$. 
When the weight is determined by the context, we will often
suppress the subindex, and write $J$ and $\Delta$ instead of 
$J_\phi$ and $\Delta_\phi$. We will also often 
identify $M$ with $\pi_\phi(M)$, and identify $\pi_{\phi}(a)$ with $a$, 
for $a\in M$.

\begin{Definition}
Let $(M, \phi)$ be  factorial noncommutative measure space. 
Let $\CH_\phi$ be the GNS space corresponding to $\phi$, and
let us identify $M$ with its image under the GNS representation
in $\CB(\CH_\phi)$. Suppose that $\theta: M \to M$ is a normal
unital endormorphism, and let $\theta_\phi': M' \to M'$ be
the endomorphism given by
\[
\theta_\phi'(y) = J_\phi \theta(J_\phi y J_\phi) J_\phi,
\qquad y \in M'
\]
We will say that $\theta$ is \emph{$\phi$-modularly extendable} if and only if there exists  a normal endomorphism  $\widetilde{\theta}_{\phi}$ of $\CB(\CH_{\phi})$ satisfying
\begin{equation}\label{ext}
\widetilde{\theta}_{\phi}(xy^{\prime}) = \theta(x) \theta_\phi'(y),
 \qquad \forall x \in M, \forall y \in M^{\prime}.
\end{equation}
where $J_\phi$ is the modular conjugation operator. We note that by
normality, such an extension is unique if it exists, and it will be
called the \emph{$\phi$-modular extension} of $\theta$.
 \end{Definition}
 
We will now show that this notion of extendability of the endomorphism $\theta$ on $M$ does not depend on the choice of weight.

\begin{Theorem}\label{inv}
Let $M$ be a factor and let $\phi$ and $\psi$ be two \fns weights on
$M$. 
\begin{enumerate}
\item\label{item:J-covariant-u} There exists a unitary $u: \CH_\phi \to \CH_\psi$  
such that
\begin{equation}\label{eq:J-covariant-u}
J_{\psi} = uJ_{\phi}u^*, \qquad \text{and} \qquad
\pi_{\psi}(x) = u\pi_{\phi}(x)u^*, \quad \forall x\in M
\end{equation}

\item\label{item:extend} Suppose that $\theta$ is a normal unital endormorphism of $M$.
Then $\theta$ is $\phi$-extendable if and only if
$\theta$ is $\psi$-extendable. In addition, if these conditions hold,
for every $u: \CH_\phi \to \CH_\psi$ satisfying
\eqref{eq:J-covariant-u}  we have
$\mathrm{Ad}(u) \circ \widetilde{\theta}_\phi =\widetilde{\theta}_\psi \circ \mathrm{Ad}(u)$.  In particular,
the modular extendability of an endomorphism does not depend on the choice of
\fns weight.
\end{enumerate}
\end{Theorem}

\begin{proof}
\eqref{item:J-covariant-u} The general theory of Hilbert/Tomita algebras,
and their (hyper)-standard forms associated to \fns weights
(see \cite[Theorem 1.14]{Taii}), yields  a unitary 
$u:\CH_\phi \to \CH_\psi$ satisfying 
\eqref{eq:J-covariant-u}.

\eqref{item:extend} Suppose that $\theta$ is $\phi$-extendable. Let $u$ be a unitary
given by item \eqref{item:J-covariant-u}. Let us consider
the endomorphism $\alpha$ of $B(\CH_\psi)$ defined by
$$
\alpha(T) = u\widetilde{\theta}_{\phi}(u^*Tu)u^*.
$$
Notice that for all $x, y\in M$, 
\begin{align*}
\alpha(\pi_{\psi}(x)J_{\psi}\pi_{\psi}(y)J_{\psi}) 
&= u\widetilde{\theta}_{\phi}(u^* \pi_{\psi}(x)J_{\psi}\pi_{\psi}(y)J_{\psi}u)u^*\\
&= u\widetilde{\theta}_{\phi}(u^* \pi_{\psi}(x)uu^*J_{\psi}uu^*\pi_{\psi}(y)uu^*J_{\psi}u^*)u^*\\
&=  u\widetilde{\theta}_{\phi}(\pi_{\phi}(x)J_{\phi}\pi_{\phi}(y)J_{\phi})u^*\\
&=  u\pi_{\phi}(\theta(x))J_{\phi}\pi_{\phi}(\theta(y))J_{\phi}u^*\\
&=  \pi_{\psi}(\theta(x))J_{\psi}\pi_{\psi}(\theta(y))J_{\psi}.
\end{align*}
Therefore $\alpha$ is a normal unital endomorphism conjugate to $\widetilde{\theta}_\phi$ extending $\theta$, proving that
$\theta$ is $\psi$-extendable and so
$\mathrm{Ad}(u) \circ \widetilde{\theta}_\phi =\widetilde{\theta}_\psi \circ \mathrm{Ad}(u)$, proving the statement.
\end{proof}

\begin{Remark}
As a consequence of Theorem \ref{inv}, since the $\phi$-modular 
extendability of an endomorphism on a factor $M$ doesn't depend on the choice of  particular f.n.s weight $\phi$ of $M$, we will speak simply of
modular extendability without reference to a weight. In \cite{ABS}, this
is called regular extendability, and in \cite{MS1}
this is called canonical extendability.
\end{Remark}

\begin{Remark}\label{rem:automorphism-extendable}
Note that if $\gamma$ is an automorphism of a factor $M$, then it
is modularly extendable. In fact, recall that from the theory of von Neumann algebras in standard form (see \cite[Theorem 1.14]{Taii}), 
 if $\phi$ is a \fns weight, then there exists
a unitary $u: \CH_\phi \to \CH_\phi $ satisfying
\begin{equation}\label{eq:conjugacy-u}
J_\phi u = u J_\phi \qquad \text{and} \qquad
\pi_\phi(\gamma (x)) = u \pi_\phi(x) 
u^*, 
\qquad \forall x \in M.
 \end{equation}
We will say that such a unitary modularly implements the automorphism
$\gamma$ with respect to $\phi$.
\end{Remark}

Furthermore, it follows immediately from the previous remark that modular extendibility is a conjugacy invariant:

\begin{Theorem}\label{thm:conjugacy}
 Suppose $\theta$ is an unital normal endomorphism of a factorial noncommutative measure space $(M, \phi)$. 
 Suppose that  $\gamma$ is an automorphism of $M$, and let  
  $\theta^{\gamma} = \gamma\circ \theta \circ \gamma^{-1}$. 
If $\theta$ is modularly extendable, then $\theta^\gamma$ is 
modularly extendable, and moreover 
  $\widetilde{\theta^\gamma}_\phi =\mathrm{Ad}(u) \widetilde{\theta}_\phi \mathrm{Ad}(u^*)$ for any unitary
  $u$ which modularly implements $\gamma$ with respect to $\phi$. 
\end{Theorem}

\section{Equimodular endomorphisms}\label{sec:equimodularity}

We now consider the concept of equimodularity of an endomorphism on a von Neumann algebra with respect to an \fns weight.  This is a generalization of the framework of \cite{BISS}, which focused
on faithful normal states. We exhibit a convenient necessary and sufficient 
condition for the equimodularity of an endomorphism with respect a fixed weight on von Neumann algebra. 

Given a noncommutative measure space $(M,\phi)$,  
let $\theta$ be a unital normal endomorphism of $M$ 
which is $\phi$-preserving, i.e.
\[
\phi(\theta(x)) = \phi(x),\qquad x \in M^+.
\]
This invariance assumption implies that there exists a unique 
well-defined isometry $u_\theta \in B(\CH_{\phi})$ given by $u_\theta (x_{\phi}) = (\theta(x))_{\phi},$  for $x \in \FN_{\phi}$.
It is clear that $u_\theta x = \theta(x) u_\theta$, for all $x \in \FN_{\phi}$. Futhermore, since $\phi$ is semi-finite, 
$\FN_{\phi}$ is dense in $M$ in the weak operator topology, 
so we have  
\[
u_\theta x = \theta(x) u_\theta, \qquad \forall x\in M.
\]

 \begin{Definition}\label{equimodular}
Given a noncommutative measure space $(M,\phi)$,
a unital endomorphism $\theta:M \to M$ will be called \emph{equimodular}
if $\phi$ is $\theta$-invariant and $u_\theta J_\phi = J_\phi u_\theta$.
 \end{Definition}
 
\begin{Remark}
The definition of equimodularity depends on the weight $\phi$. So strictly speaking,  we should call it $\phi$-equimodular, but we shall 
not do so in the interest of  notational convenience.   
Furthemore, we note that this definition applies also to weights, 
in a slight generalization of \cite{BISS}.
\end{Remark}

\begin{Example}
 Any unital normal endomorphism on a II$_1$ factor is equimodular with respect to the trace  (see \cite{BISS}).
\end{Example}

Given a unital normal endomorphism $\theta$ on a factor, we now describe a necessary and sufficient condition for the existence of an \fns weight $\phi$ with respect to which $\theta$ is equimodular. As a consequence, we
exhibit endomorphisms which are not equimodular with respect to
any \fns weight.

The necessary part of the following theorem was proven
for equimodularity with respect to \emph{states} 
 in   \cite[Remark 3.2]{BISS}. Here we extend the original
 proof of necessity to general weights, and moreover prove also sufficiency of the criterion.

\begin{Theorem}\label{equi}
Let $(M, \phi)$ be a non-commutative measure space. Suppose 
$\theta$ is a unital normal  endomorphism on $M$
which is $\phi$-preserving, i.e.\ $\phi(\theta(x)) = \phi(x)$ for all 
 $x \in M^{+}$. Then $\theta$
 is equimodular if and only if there exists a faithful normal conditional expectation 
 $E: M \rightarrow \theta(M)$ which is $\phi$-preserving, i.e.\  
 \[ 
 \phi(E(x)) = \phi(x),\qquad \forall x\in M^+.
 \]
Furthermore, such a conditional expectation is unique if it exists since
$E(x) e_\theta =  e_\theta x e_\theta$ for all
 $x\in M$, where $e_\theta$ is the projection onto the closure of $(\theta(M)\cap \FN_\phi)_\phi$, and moreover $e_\theta=u_\theta u_\theta^*$.
\end{Theorem}
 \begin{proof}
 Firstly, let $\theta$ be a unital normal endomorphism of $M$ which is
 $\phi$-preserving.   Let 
 $\FA_{\phi} = \FN_{\phi} \cap \FN_{\phi}^{*} \subseteq \CH_{\phi}$
be the left Hilbert algebra associated to 
the \fns weight $\phi$ of $M$ (see \cite[2.11, p.24]{SS}) and 
let 
 $S_{\phi}$ be the corresponding anti-linear 
  Tomita operator on $\CH_{\phi}$. 
 For $x_{\phi} \in \FA_{\phi} \subseteq D(S_{\phi})$, observe that
 $u_\theta x_\phi = (\theta(x))_\phi \in D(S_\phi)$ and 
 \begin{align*}
 S_{\phi}u_{\theta} x_{\phi} &=  S_{\phi}(\theta (x))_{\phi} = (\theta (x^*))_{\phi} = u_{\theta}(x^*)_{\phi} = u_{\theta} S_{\phi}x_{\phi}.
 \end{align*} 
 So we conclude that   
\begin{equation}\label{eq:S-phi}
S_{\phi}u_{\theta} \xi = u_{\theta} S_{\phi} \xi, \qquad \forall \xi \in D(S_{\phi}).
\end{equation}

$(\Rightarrow)$ Suppose $\theta$ is equimodular. Since $u_\theta$ commutes with $J_\phi$ and $S_{\phi}=J_{\phi}\Delta_{\phi}^\half$, 
we have $\Delta_{\phi}^\half u_\theta =  u_\theta \Delta_{\phi}^\half$ 
on $D(\Delta_\phi^\half)$, 
and so $\Delta_{\phi}^{it}$ commutes with $u_\theta$ for all $t\in \R$. 
Hence, for all $x \in \FN_{\phi}$ and $t \in \R$, 
by \cite[Chapter 1, Section 2.12]{SS} we have
\begin{align*}
(\theta\circ\sigma^\phi_t(x))_{\phi} &= u_\theta (\Delta_{\phi}^{it} x \Delta_{\phi}^{-it}) _{\phi}
= u_{\theta} \Delta_{\phi}^{it} x_{\phi}
= \Delta_{\phi}^{it} u_\theta x_{\phi}  
= \Delta_{\phi}^{it} (\theta(x))_{\phi} \\
& = (\sigma^\phi_t \circ\theta(x)) _{\phi}.
\end{align*}
Since $\phi$ is faithful, we conclude  that 
\[
\theta\circ\sigma^\phi_t(x) = \sigma^\phi_t \circ\theta(x),
\qquad  \forall x \in \FN_{\phi}.
\] 
Now semi-finiteness of $\phi$ implies that  
\[
\theta\circ\sigma^\phi_t(x) = \sigma^\phi_t \circ\theta(x),
\qquad \forall x \in M.
\]
Furthermore, the fact that $\theta$ is $\phi$-preserving implies that $\phi|_{\theta(M)}$ is also 
semi-finite for $\theta(M)$. It follows from Takesaki's theorem (see \cite[Section 3, p.~309]{Ta}) 
that there exists a unique $\phi$-preserving
 conditional expectation $E$ of $M$ onto  $\theta(M)$. 
Moreover, it follows from \cite[p. 315]{Ta} or \cite[p.131]{SS}  that $E$ is unique and satisfies $E(x)e_\theta = e_\theta x e_\theta$ for all $x \in M$, where $e_\theta = u_\theta u_{\theta}^*$ is the projection onto the closure of $(\theta(M)\cap \FN_\phi)_\phi$.

($\Leftarrow$) Conversely, let us assume that there exists a normal conditional 
expectation $E$ from  $M$ onto $\theta(M)$ which is $\phi$-preserving. We need to prove that $\theta$ is equimodular, i.e.\ $u_{\theta}J_{\phi} = J_{\phi} u_{\theta}$. 
By eq.~\eqref{eq:S-phi}, $u_{\theta} S_{\phi} = S_{\phi}u_{\theta}$
on $D(S_\phi)= D(\Delta_\phi^\half)$. Hence it is enough 
 to show that $u_{\theta}\Delta_{\phi}^\half =  \Delta_{\phi}^\half u_{\theta} $ on $D(\Delta_\phi^\half)$. The following lemma 
 will complete the proof. 
 \end{proof}
 
 \begin{Lemma}
 Let $(M,\phi)$ be a noncommutative measure space. 
 Suppose that $\theta$ is a unital $\phi$-preserving normal endomorphism on $M$ such that there exists 
 a faithful normal conditional expectation $E$ from $M$ onto $\theta(M)$ which is $\phi$-preserving.  Then  for all $t\in \R$, 
\begin{enumerate}\addtolength{\itemsep}{0.4cm}
 \item\label{lemma-cond1}
 $ \theta \circ \sigma^{\phi}_t = \sigma_t^{\phi} \circ \theta$

 \item\label{lemma-cond2} $u_\theta D(\Delta_\phi^\half) \subseteq 
 D(\Delta_\phi^\half)$ and 
 $u_{\theta}\Delta_{\phi}^\half = \Delta_{\phi}^\half u_{\theta} $
 on $D(\Delta_\phi^\half)$.
\end{enumerate}
 \end{Lemma}
\begin{proof}
 Let $\psi= \phi|_{\theta(M)}$, which is an \fns weight on 
 $\theta(M)$ and let  $\CH_\psi$ be its GNS Hilbert space. Under the
 the conditions of the lemma, by Takesaki's Theorem 
 \cite[Section 3, p.~309]{Ta} we have that 
 $\sigma^{\phi}_t(\theta(M)) = \theta(M)$. Furthermore, note that
\begin{equation}\label{restriction-modular}
\sigma^{\phi}_t|_{\theta(M)}= \sigma^{\psi}_t,~t \in \R.
\end{equation}
Let $u_\theta\in B(\CH_\phi)$ be given by 
 $u_\theta (x_{\phi}) = (\theta(x))_{\phi} = (\theta(x))_{\psi}$,
 since $\psi$ is the restriction of $\phi$ to $\theta(M)$. With
 this identification, $u_\theta$ becomes a \emph{unitary}
 map from $u_\theta: \CH_\phi \to \CH_\psi$ such that 
 $\theta(x) = u_{\theta}xu_{\theta}^*$, for all $x \in M$.

Now let us consider $\pi_t : \theta(M) \to \theta(M)$ given by 
\[
\pi_t(y) = u_{\theta} \Delta_{\phi}^{it} u_{\theta}^* y u_{\theta}\Delta_{\phi}^{-it}u_{\theta}^*, \qquad y \in \theta(M).
\]
We note that $\{\pi_t : t\in\R\}$ defines a group 
 of $*$-automorphisms of $\theta(M)$ which is $\psi$-preserving. Indeed, 
 for every $x \in \theta(M)^{+}$, since $\theta$
 is $\phi$-preserving we have that 
 \begin{align*}
 \psi(\pi_t(\theta(x))) 
 &=\phi(u_{\theta}\Delta_{\phi}^{it} u_{\theta}^*\theta(x) u_{\theta}\Delta_{\phi}^{-it}u_{\theta}^*)
 =  \phi(u_{\theta} \sigma_t^{\phi}(x)u_{\theta}^*)\\
 &=  \phi(\theta(\sigma_t^{\phi}(x))) =  \phi(\sigma_t^{\phi}(x)) \\
 & =  \phi(x) = \phi(\theta(x))= \psi(\theta(x))
 \end{align*}
Since $\phi$ satisfies 
the KMS condition with respect to  $\{\sigma^{\phi}_t\}$, 
given $x, y \in \FN_{\phi} \cap \FN_{\phi}^{*}$,
 there exists a bounded continuous function  
 $f_{x,y} : \{ z\in\C : 0 \leq \re z \leq 1\} \to \C$
 which is analytic in the interior of the strip
 such that for all $t\in \R$,
\[
 f_{x,y}(it) = \phi(x\sigma^{\phi}_t(y)), 
 \qquad  \text{and} \qquad
 f_{x,y}(1+it) = \phi(\sigma^{\phi}_t(y)x).
 \]
We now claim that $\psi$ satisfies the KMS condition with 
respect to $\{\pi_t\}$. Indeed, given any pair 
$a, b \in \FN_\psi\cap \FN_\psi^* (\subseteq \theta(M))$, 
there exist unique $x,y \in M$ such
that $a = \theta(x), b=\theta(y)$. Since $\theta$ is $\phi$-preserving, we have that $x,y \in \FN_\phi \cap \FN_\phi^*$.
Hence the function
$\widetilde{f}_{a,b} : \{ z\in\C : 0 \leq \re z \leq 1\} \to \C$
given by  $\widetilde{f}_{a,b} = f_{x,y}$ is well-defined. Moreover, for all $t\in\R$,
\begin{align*}
\psi(a\pi_t(b)) & =
\psi(\theta(x)\pi_t(\theta(y))) 
= \psi(\theta(x)u_{\theta}\Delta_{\phi}^{it} u_{\theta}^*\theta(y) u_{\theta}\Delta_{\phi}^{-it}u_{\theta}^*) \\
& = \phi(\theta(x)u_{\theta} \sigma_t^{\phi}(y)u_{\theta}^*) 
 = \phi(u_{\theta}xu_{\theta}^*u_{\theta} \sigma_t^{\phi}(y)u_{\theta}^*)
= \phi(u_{\theta}x\sigma_t^{\phi}(y)u_{\theta}^*) \\
& = \phi(\theta(x\sigma_t^{\phi}(y))) 
= \phi(x\sigma_t^{\phi}(y)) = f_{x,y}(it) \\
& = \widetilde{f}_{a,b}(it)
\end{align*}
and
\begin{align*}
\psi(\pi_t(b) a) & =
\psi(\pi_{t}(\theta(y))\theta(x))
= \psi(u_{\theta}\Delta_{\phi}^{it} 
u_{\theta}^*\theta(y) u_{\theta}\Delta_{\phi}^{it}u_{\theta}^*\theta(x)) \\
& = \phi(u_{\theta} \sigma_t^{\phi}(y)u_{\theta}^*\theta(x))
= \phi(u_{\theta} \sigma_t^{\phi}(y)u_{\theta}^*u_{\theta}xu_{\theta}^*)
= \phi(u_{\theta} \sigma_t^{\phi}(y)xu_{\theta}^*)\\
&= \phi(\theta (\sigma_t^{\phi}(y)x))
= \phi(\sigma_t^{\phi}(y)x) = f_{x,y}(1+it)\\
& = \widetilde{f}_{a,b}(1+it).
\end{align*}
Hence, by the uniqueness of the modular automorphism group, we have that
\begin{equation}\label{eq:pi-is-sigma}
 \sigma^{\psi}_t = \pi_t,  \qquad t \in \R.
\end{equation}
We are now in position to prove property~\eqref{lemma-cond1}. 
For all $x \in M$, by \eqref{restriction-modular} and 
\eqref{eq:pi-is-sigma} we have 
\begin{align*}
\theta(\sigma_t^{\phi}(x)) 
&= u_{\theta} \sigma_t^{\phi}(x)u_{\theta}^*
= u_{\theta}\Delta_{\phi}^{it}x \Delta_{\phi}^{-it}u_{\theta}^*\\
&= u_{\theta}\Delta_{\phi}^{it}u_{\theta}^*u_{\theta}xu_{\theta}^*u_{\theta} \Delta_{\phi}^{-it}u_{\theta}^* \\
& = \pi_t(\theta(x))
= \sigma^{\psi}_t(\theta(x))\\
&= \sigma^{\phi}_t(\theta(x)).
\end{align*}

We now prove property \eqref{lemma-cond2}. Recall that for all $x\in \FN_\phi$ and $t\in \R$, we have that
$\Delta^{it} x_\phi = (\sigma_t^\phi(x))_\phi$ (see for example
\cite[p.27]{SS}). Therefore, by property \eqref{lemma-cond1},
for all $x \in \FN_{\phi}$, we have that
\begin{align*}
u_{\theta}\Delta_{\phi}^{it}x_{\phi}
& = u_\theta (\sigma_t^{\phi}(x))_{\phi} 
=(\theta (\sigma_t^{\phi}(x)))_{\phi}
=(\sigma_t^{\phi}(\theta(x)))_{\phi}\\
&=\Delta_{\phi}^{it} (\theta(x) )_{\phi}
=\Delta_{\phi}^{it}u_{\theta}x _{\phi}.
\end{align*}
Thus, 
\begin{equation}\label{eq:u-commutes-delta}
u_{\theta}\Delta_{\phi}^{it} = \Delta_{\phi}^{it}u_{\theta}.  
\end{equation}
Notice that if $x\in \FN_\phi\cap \FN_\phi^*$ then also 
$\theta(x) \in \FN_\phi\cap \FN_\phi^*$ and therefore
$x_\phi \in D(\Delta_\phi^\half)$ and  $u_\theta x_\phi = (\theta(x))_\phi \in D(\Delta_\phi^\half)$. Hence, it follows from
\cite[Corollary 9.21]{SZ}
that  
both maps $ it \rightarrow \Delta_{\phi}^{it}u_{\theta} x_{\phi}$ and
$ it \rightarrow u_{\theta}\Delta_{\phi}^{it}x_{\phi}$ 
have continuous extensions to the 
strip $\{ z\in \C :  0 \leq \text{Re}(z) \leq 1/2 \}$ which 
are analytic in the interior. Furthermore, by 
\eqref{eq:u-commutes-delta}, the two functions
coincide on the imaginary axis, therefore they coincide on
the entire strip. (This follows from a standard argument using
the Schwarz reflection principle and gluing the reflection to
the original function). Hence, by  taking $ z = \half$, we 
conclude that $u_{\theta}\Delta_{\phi}^{\half} x_\phi = \Delta_{\phi}^{\half}u_{\theta} x_\phi$. 
Therefore $u_\theta D(\Delta_\phi^\half) \subseteq D(\Delta_\phi^\half)$
and $u_{\theta}\Delta_{\phi}^{\half} = \Delta_{\phi}^{\half}u_{\theta}$
on $D(\Delta_\phi^\half)$.
\end{proof}

 \begin{Remark}\label{not-image-expectation}
The necessary and sufficient condition provided by 
Theorem~\ref{equi} is quite restrictive. In fact, it is
easy to find endomorphisms whose range is not the image
of a normal conditional expectation, without referring to
the preservation of any weight. For example, let 
$\CH$ be an infinite dimensional separable Hilbert space
and let $M\subseteq \CB(\CH)$ be a type III factor.
  Note that  the exists a *-isomorphism 
  $\gamma: M\otimes \CB(\CH) \rightarrow M$, since $M$ is a type III factor.  Let $\iota: M \rightarrow 1\otimes M \subseteq M\otimes\CB(\CH)$ be the canonical injection, i.e.\
 $\iota(x) = 1\otimes x$, for $x \in M$. 
 Then $ \theta = \iota \circ \gamma: M\otimes\CB(\CH) \rightarrow M\otimes\CB(\CH)$ 
 is an endomorphism of $M\otimes\CB(\CH)$ such that
 $\theta(M\otimes \CB(\CH))=1\otimes M$. 
Note, however, that a normal (surjective) conditional expectation 
$E: M \otimes \CB(\CH) \to 1 \otimes M$ would restrict to a normal
conditional expectation   
$E|_{1 \otimes \CB(\CH)}: 1 \otimes \CB(\CH) \rightarrow 1 \otimes M$.
And such a conditional expectation can only exist if $M$ is type I
(see \cite{tomiyama-III}).
 \end{Remark}

\begin{Example}\label{ex:not-equimodular}
We use the principle of the previous remark to provide an example of endomorphism which is not equimodular with respect to any \fns weights.
Let $\CH$ be an infinite dimensional separable Hilbert space,  let 
 $M \subseteq \CB(\CH)$ be a type III factor with cyclic 
 separating vector $\Omega \in \CH$, and let $N \subseteq M$ be
 a subfactor such that there is no normal conditional expectation
 of $M$ onto $N$ (for an example, 
 see Remark~\ref{not-image-expectation}).  
  Let us consider the infinite
 tensor product  $K=\bigotimes_{n=1}^\infty (H, \Omega)$ 
 with respect to the reference vector  $\Omega^\infty = \bigotimes_{n=1}^\infty \Omega$. Let $A = N \otimes M \otimes M \dots 
 \subseteq B(K)$, in other words $A$ is the factor generated by
 elements of the form $n \otimes m_1 \otimes m_2 \otimes \cdots \otimes m_k \otimes I \otimes I \cdots$ where $n\in N$ and $m_j \in M$ for
 $j=1,\dots, k$ (see \cite{guichardet-infinite-tensor-products}). Note
 that we can write $A = N \otimes M \otimes B$ where $B= \bigotimes_{n=1}^\infty M$ in the natural way.  
  
 Let $\theta : A \rightarrow A$ be the shift endomorphism, so that
 $\theta(A) = 1 \otimes N \otimes B$. It is clear that $\theta$ preserves the vector
 state corresponding to $\Omega^\infty$ (in particular, it preserves \emph{some} \fns weight). Suppose towards a contradiction 
 that $\theta$  is equimodular with respect to some \fns weight $\phi$.
 Then by Theorem~\ref{equi} there exists a normal conditional expectation
 $E$ from $A$ onto $\theta(A)$ which is $\phi$-preserving. In that 
 case, we claim that $E(1_N \otimes M \otimes 1_B ) = 1_N \otimes N \otimes 1_B$, which is a contradiction, since there is no normal conditional  expectation from $M$ onto $N$. Indeed, notice that
 $1 \otimes M \otimes 1 = A \cap (N\otimes 1 \otimes B)'$. So if 
 $x\in 1\otimes M \otimes 1$ and $y\in 1 \otimes 1 \otimes B \subseteq  \theta(A)$, then
 \[
 E(x) y = E(xy) = E(yx) = y E(x)
 \]
 Thus $E(x) \in (1 \otimes 1 \otimes B)'\cap \theta(A) = 1 \otimes N \otimes 1$.
\end{Example}

We now provide the promised convenient condition for the modular
extendibility of equimodular endomorphisms.

\begin{Theorem}\label{extcon}
Let $(M, \phi)$ be a noncommutative factorial measure space and let $\theta$
be a unital normal endomorphism of $M$ which is equimodular. If 
\[
(\theta(M) \cup (M \cap \theta(M)'  )) '' = M
\]
then $\theta$ is modularly extendable.
\end{Theorem}

\begin{proof}
Let $N= \theta(M)$, which is a factor since $M$ is a factor and $\theta$
is normal. Since $\theta$ is equimodular, by 
Theorem~\ref{equi} there exists a unique $\phi$-preserving 
normal conditional expectation $E: M \rightarrow N$.
Suppose that $(N \cup (M \cap N')) '' = M$. We will show that $\theta$ has a (unique) modular extension.
For $x\in M\cap N^\prime$ and $a\in N$
we observe that 
\[ 
E(x)a = E(xa) = E(ax)= aE(x).
\] 
So for $x \in M\cap N^\prime$, we have $E(x) \in N \cap N^{\prime}$. Thus, since $N$ is a factor, $E(x)$ is a scalar multiple of the identity. Now if $x,y \in M\cap N^\prime$, then 
there exists a unique scalar $\langle x, y \rangle_E \in \C$
such that
we define $E(y^*x) = \langle x, y \rangle_E \cdot 1$ (here $1$ is the identity element of $M$ as well as $N$ and $M\cap N^\prime$) 
and we check that $\langle\cdot , \cdot\rangle_E$ defines a 
inner product on $M\cap N^\prime$. 
Notice that  $\langle\cdot , \cdot\rangle_E$ is clearly an 
inner product on $M\cap N'$ since $E$ is a faithful conditional
expectation. 

Let $\CK_{E}$ be the Hilbert space obtained by completion of $M\cap N^\prime$  with respect to $\langle \cdot, \cdot \rangle_E$.
Now we notice that there exists a unitary  $W: \CK_E \otimes \CH_{\phi} \rightarrow \CH_{\phi}$ satisfying 
\[
W(x\otimes a_\phi) =  (x \theta(a))_\phi, \qquad x\in M\cap N' , a \in \FN_\phi \cap \FN_\phi^*
\]
Indeed, if $x,y \in M\cap N^\prime$ and $a,b \in \FN_\phi \cap \FN_\phi^*$,  we have that
\begin{align*}
\langle (x \theta(a))_\phi,   (y \theta(b))_\phi \rangle & =
\phi(\theta(b)^* y^* x \theta(a) ) 
= \phi(E( \theta(b)^* y^* x \theta(a) )) \\
& = \phi( \theta(b)^* E(y^* x) \theta(a) ) 
= \langle x, y \rangle_E \cdot \phi( \theta(b)^*  \theta(a) )  \\
& = \langle x, y \rangle_E \cdot \phi( \theta(b^*a) )
= \langle x, y \rangle_E \cdot \phi( b^*a ) \\
& = \langle x, y \rangle_E \cdot 
\langle a_\phi, b_\phi \rangle
\end{align*}
therefore $W$ is a well-defined isometry. Moreover, $W$ is
surjective since $(N^\prime \cap M)\cup N)'' = M$, and hence
the elements of the form $x \theta(a)$ for $x \in M\cap N'$ and
$a \in \FN_\phi\cap \FN_\phi^*$ are weakly dense in $M$. We notice
that $W$ also satisfies the property that
\[
W( x\otimes \xi) = x u_\theta \xi, \qquad x\in M\cap N', \xi \in \CH_\phi
\]

We now consider the unital endomorphism $\alpha$ of $B(\CH_\phi)$ given by
\[
\alpha(x) = W (1 \otimes x )W^*, \qquad x\in B(\CH_\phi).
\]
We claim that $\alpha$ is the modular extension of $\theta$. Indeed,
given  $x \in M$, we have that for all $ y\in N^\prime \cap M , a \in \FN_\phi \cap \FN_\phi^*$, 
 \begin{align*}
\alpha(x) W y \otimes a_\phi 
& = W (1 \otimes x) (y \otimes a_\phi)
 = W y \otimes (xa)_\phi =  (y\theta(xa))_\phi \\
 & = (\theta(x) y \theta(a))_\phi
 = \theta(x) W ( y \otimes a_\phi )
 \end{align*}
hence $\alpha=\theta$ on $M$. Similarly, by equimodularity
 \begin{align*}
\alpha(J_\phi x J_\phi) W y \otimes a_\phi 
& = W (1 \otimes J_\phi xJ_\phi) (y \otimes a_\phi)
 = W y \otimes J_\phi xJ_\phi a_\phi \\
 & =  y u_\theta J_\phi xJ_\phi a_\phi 
=  y  J_\phi u_\theta xJ_\phi a_\phi 
=  y  J_\phi  \theta(x) u_\theta J_\phi a_\phi \\
& =  y  J_\phi  \theta(x) J_\phi u_\theta  a_\phi 
=  y  J_\phi  \theta(x) J_\phi (\theta(a))_\phi \\
& =   J_\phi  \theta(x) J_\phi (y \theta(a))_\phi 
= J_\phi  \theta(x) J_\phi W (y\otimes a_\phi)
 \end{align*}
hence $\alpha(J_\phi x J_\phi) = J \theta(x) J$ for $x\in M$. 
Hence $\alpha$ is the modular extension of $\theta$. 
\end{proof}

\begin{Remark}\label{stateext}
It is unclear to us whether the converse is true in the case
of \fns weights. It certainly holds in the case of faithful normal 
\emph{states}, by \cite[Corollary 3.7]{BISS}. 
\end{Remark}

\begin{Remark} \label{explicit-ext}
Let $(M, \phi)$ be a noncommutative measure space and let $\theta$
be an equimodular unital endomorphism of $M$. 
In the particular case when $\phi$ is a faithful normal state, the
modular extension $\alpha$ of $\theta$ guaranteed by Theorem~\ref{extcon} can be made very explicit. 
Recall that in the proof we exhibit
$W: \CK_E \otimes \CH_\phi \to \CH_\phi$ such that 
$\alpha(x) = W (1 \otimes x) W^*$ for all $x \in \CB(\CH_\phi)$
(here we are using the same notation). 
When $\phi$ is a faithful normal state, the GNS representation
on $\CH_\phi$ has a cyclic separating vector $\Omega$ (corresponding
to the identity). Furthermore, the space $\CK_E$ can be identified
explicitly with a subspace of $\CH_\phi$ via the isometry 
$\CK_E \to \CH_\phi$ given by $x \mapsto x\Omega$, for 
$x \in M\cap \theta(M)'$. Hence using this identification, we have that 
$W: \CK_E \otimes \CH_\phi \to \CH_\phi$ is given by
\[
W(x\Omega \otimes y\Omega) = x \theta(y) \Omega, \qquad x \in M\cap \theta(M)' , y \in M
\]
\end{Remark}

\section{E$_0$-semigroups on factors}\label{sec:E-semigroups}

In this section we study modular extendability for 
$E_0$-semigroups on arbitrary factors. This leads
to a classification scheme for $E_0$-semigroups based
on the well-established classification of E$_0$-semigroups
on type I$_\infty$ factors due to Powers and Arveson, as well
as several cocycle-conjugacy invariants.

We note that the study of modular extendability
with respect to faithful states, as opposed to \fns weights, 
was the main focus of \cite{BISS}.

Recall that an \emph{E$_0$-semigroup} on a W*-algebra
$M$ is a family $\alpha = \{\alpha_t: t \geq 0\}$
of  normal unital *-homomorphism of $M$ such that
$\alpha_0 = id_M$ and $\alpha_s \circ \alpha_t = \alpha_{s+t}$ for
all $t,s \geq 0$ which is weak*-continuous, i.e. for every
$\rho \in M_*$ and $x\in M$, the map $[0,\infty) \ni t \mapsto \rho(\alpha_t(x))$ is continuous.

We will be interested in the classification of $E_0$-semigroups
of a von Neumann algebra $M$ with respect to \emph{cocycle 
conjugacy}, which we review presently. 

Let $M$ be a von Neumann algebra. Given an E$_0$-semigroup $\alpha$ 
on $M$, a strongly continuous family of unitary 
$ U = \{ U_t : t \geq 0 \}$ in $M$ will be called an $\alpha$-cocycle
if $U_{s+t} = U_t \alpha_t(U_s)$, for all $s,t \geq 0$. Notice that
for an $\alpha$-cocycle $U$ we automatically have $U_0 = I$. 
An E$_0$-semigroup $\beta$ on $M$ is said to be \emph{conjugate}
to $\alpha$ if there exists  an automorphism $\gamma \in \mathrm{Aut}\, M$ such that 
\[
\gamma \circ \beta_t \circ \gamma^{-1} = \alpha_t, \qquad
t\geq 0.
\]
We will say that $\beta$ is  \emph{cocycle equivalent} to $\alpha$ if there exists an 
$\alpha$-cocycle $U$ such that
\[ 
\beta_t(x) = U_t \alpha_t(x) U_t^* , \qquad  t\geq 0, x\in M
\]
Finally, we will say that $\beta$ is \emph{cocycle conjugate} to $\alpha$
if there exists an E$_0$-semigroup $\beta'$ of $M$ which is conjugate
to $\beta$ such that $\beta'$ is cocycle equivalent to $\alpha$.

We now bring the concepts of the last sections concerning 
extendability to the context of E$_0$-semigroups. 

\begin{Definition}
Let $M$ be a von Neumann algebra. An E$_0$-semigroup 
$\alpha$ of $M$ will be called \emph{extendable} if there exists
a normal unital nondegenerate representation $\pi:M \to \CB(\CH)$ 
and E$_0$-semigroup $\widetilde{\alpha}$ on $\CB(\CH)$ such that
$$
\widetilde{\alpha}_t(\pi(x)) = \pi(\alpha_t(x)), \qquad \forall x\in M, t\geq 0.
$$
\end{Definition}

This definition is very natural, however \emph{all} E$_0$-semigroups
on von Neumann algebras with \emph{separable predual} are extendable
by a result of Arveson and Kishimoto~\cite{arv-kishimoto}. In fact,
their result states that we can find a pair $(\pi, \widetilde{\alpha})$
as above so that $\widetilde{\alpha}$ is a semigroup of automorphisms. 
We will focus
on a more restrictive concept of extendability as introduced in 
Section~\ref{sec:extendibility}.
 
\begin{Definition}
An $E_0$-semigroup $\alpha$ on a factor $M$ is said to \emph{modularly extendable} if
$\alpha_t$ is modularly extendable for every $t\geq 0$.
\end{Definition}

We note that by Theorem~\ref{inv}, modular extendibility of 
endomophisms on $M$ can be checked with respect to \emph{any} 
\fns weight on $M$.

\begin{Definition}
An $E_0$-semigroup $\alpha = \{\alpha_t: t \geq 0\}$ on a factorial noncommutative measure space $(M, \phi)$ is said to be 
\emph{equimodular} if every $t\geq 0$,  $\alpha_t$ is an equimodular endomorphism with respect to $\phi$.
\end{Definition}
 
The following theorem is an immediate consequence of Theorem~\ref{extcon} and  \cite[Corollary 3.7]{BISS} to
the context of E$_0$-semigroups.
 
\begin{Theorem}\label{E0-extcon}
Let $(M, \phi)$ be a noncommutative factorial measure space and let 
$\alpha$ be an E$_0$-semigroup on $M$ which is equimodular. If 
\[
(\alpha_t(M) \cup (M \cap \alpha_t(M)'  )) '' = M, \qquad \forall t\geq 0
\]
then $\alpha$ is modularly extendable. Furthermore, when $\phi$ is
a state, the converse holds.
\end{Theorem}

\subsection{Type classification of E$_0$-semigroups}
We are naturally led to the following classification scheme.

\begin{Definition}
Let $M$ be a factor, and let $\alpha$ be an E$_0$-semigroup
on $M$. If $\alpha$
has modular extension $\widetilde{\alpha}$ on some $\CB(\CH_\phi)$ 
for some \fns weight $\phi$ (and hence all \fns weights
by Theorem~\ref{inv}) , we will say
that $\alpha$ has type EI, EII or EIII, respectively if
$\widetilde{\alpha}$ has type I, II or III, respectively, in
the sense of Arveson and Powers. Otherwise, we will simply say 
that $\alpha$ is not modularly extendable.
\end{Definition}

\begin{Remark}\label{einv}
Notice that the type of the modular extension does not depend on the choice of \fns weight. Indeed, let $M$ a factor and let $\alpha = \{\alpha_t: t \geq 0\}$ be
an extendable  $E_0$-semigroup. Suppose that $\phi$ and   $\psi$ are 
two \fns 
weights on $M$ with corresponding modular extensions  $\widetilde{\alpha}_\phi$ and $\widetilde{\alpha}_\psi$ 
on $\CB(\CH_\phi)$ and $\CB(\CH_\psi)$ respectively. 
By Theorem~\ref{inv}, replacing $\theta$ by $\alpha_t$  for all 
$t \geq 0$, and choosing a particular unitary 
 $u : \CH_\phi \rightarrow \CH_\psi $
satisfying 
\eqref{eq:J-covariant-u}
we obtain
 $$ \text{Ad}(u)\circ (\widetilde{\alpha_t})_{\phi} \circ\text{Ad}(u^*)= (\widetilde{\alpha_t})_{\psi}, \qquad \forall t\geq 0,
 $$
i.e $\widetilde{\alpha}_{\phi}$ and $\widetilde{\alpha}_{\psi}$ are conjugate and hence possess the same type.
\end{Remark}

Now we observe that the type of an $E_0$-semigroup on a factor $M$ is a 
cocycle conjugacy invariant. 

\begin{Proposition}\label{prop:cocycle-conjugacy}
Let $M$ be a factor and let $\alpha$ and $\beta$ be cocycle
conjugate E$_0$-semigroups on $M$. Suppose that $\alpha$ is modularly
extendable. Then $\beta$ is modularly extendable, and moreover, the
modular extensions of $\alpha$ and $\beta$ are cocycle conjugate.
Therefore, the type of an $E_0$-semigroup on $M$ is a 
cocycle conjugacy invariant. 
\end{Proposition}
\begin{proof} Suppose that $\gamma$ is an automorphism of $M$
and $(U_t)_{t\geq 0}$ is an $\alpha$-cocycle such that
\[
\gamma(\beta_t(\gamma^{-1}(x))) = U_t\alpha_t(x)U_t^*, \qquad \forall x\ in M, \forall t\geq 0.
\]
By Remark~\ref{rem:automorphism-extendable}, there exists a unitary
$W:\CH_\phi \to \CH_\phi$ which modularly implements $\gamma$ with
respect to $\phi$. Let $\wgamma = \mathrm{Ad}(W)$. Let 
$\walpha_t$ be the modular extension of $\alpha_t$
with respect to $\phi$, and let 
$V_t = U_t J U_t J$, for every $t \geq 0$. Then 
it is straightforward to check that $(V_t)_{t\geq0}$ is an $\walpha$-cocycle, and moreover the E$_0$-semigroup of $\CB(\CH_\phi)$
defined by
$$
x \quad \mapsto \quad \wgamma^{-1}(V_t \walpha_t(\wgamma(x) V_t^* )    )
$$
is the modular extension of $\beta$ and it is clearly cocycle 
conjugate to $\walpha$.
\end{proof}

\begin{Remark}\label{rem:classification-in-BH}
We note that when $M$ is a type I$_\infty$ factor, i.e. $M$ is isomorphic to $\CB(\CH)$ for some separable Hilbert space $\CH$,
the type of E$_0$-semigroups generalizes the classification of
Arveson and Powers. More precisely, if $\alpha$ is an E$_0$-semigroup
of $\CB(\CH)$, then
\begin{enumerate}
\item $\alpha$ is modularly extendable.
\item $\alpha$ has type EI, EII or EIII, respectively, 
if and only if $\alpha$ has type I, II or III, respectively, in
the sense of Arveson and Powers.
\end{enumerate}
Indeed, let $\CH$ be a separable Hilbert space and let $\alpha$
be an E$_0$-semigroup on $\CB(\CH)$. Let $(e_n)_{n\geq 1}$ be
an orthonormal basis for $\CH$ and let $(\lambda_n)_{\geq 0}$ be
a sequence nonzero positive real numbers such that
\[
\sum_{n=1}^\infty \lambda_n^2 = 1
\]
Then we obtain a faithful normal state on $\CB(\CH)$ defined by
$$
\phi(x) = \sum_{n=1}^\infty \lambda_n^2 \langle x e_n, e_n \rangle
$$
We will show that $\alpha$ is modularly extendable with respect to
$\phi$. 

Let $Q: \CH \to \CH$ be the conjugate linear self-adjoint unitary given by $Qe_n = e_n$ for all $n$. Consider the vector 
$$
\Omega = \sum_{n=1}^\infty \lambda_n e_n \otimes e_n
$$
and let $\pi: \CB(\CH) \to \CB(\CH \otimes \CH)$ be the normal 
representation given by $\pi(x) = x \otimes 1$. It is easy to check
that $(\CH \otimes \CH, \pi, \Omega)$ is a GNS tripe for $\phi$.
Furthermore, let $\Delta$ and $J$ be the modular operator and modular conjugation operators, which have $\pi(M)\Omega$ as a core. This core
contains
all vectors of the form $e_m \otimes e_n$ for $m,n \geq 1$ and it is straightforward to check that
\begin{align*}
\Delta^\half (e_m \otimes e_n) & = \frac{\lambda_m}{\lambda_n}\, 
e_m\otimes e_n, \qquad m,n\geq 1 \\
J( \xi \otimes \eta ) & = Q\eta \otimes Q\xi, \qquad \xi,\eta \in \CH
\end{align*}
Let $\beta$ be the E$_0$-semigroup of $\CB(\CH)$ given by 
$\beta_t(x) = Q\alpha_t(QxQ)Q$ for $t\geq 0$. It is easy to check
that $\alpha \otimes \beta$ on $\CB(\CH\otimes \CH)$ is the modular
extension of $\alpha$. Furthermore, if $E^\alpha$ and $E^\beta$ are
the product systems associated to $\alpha$ and $\beta$, respectively,
then the map $\Theta: E^\alpha \to E^\beta$ given by $\Theta(T) = QTQ$
for every $T \in E^\alpha(t)$ is a conjugate linear isomorphism of
product systems. Therefore, they share the same type and index. 
In particular, we obtain that $\alpha$ and its modular extension 
have the same type, and the index is doubled.
\end{Remark}

In Section~\ref{sec:examples} we will give more concrete examples on factors of types II$_1$, II$_\infty$ and III. 
We note that automorphism groups are modularly extendable, hence we always have trivial examples of E$_0$-semigroups of type EI on every factor.
At present it is still unclear to us whether any II$_1$ 
factor has nontrivial E$_0$-semigroups of types EI and EII (of course type EIII is impossible in this case). In the properly infinite case, we can say a bit more.

\begin{Theorem}
Suppose that $\alpha$ and $\beta$ are E$_0$-semigroups on factors
$M$ and $N$, respectively. Then 
$\alpha \otimes \beta$ is modularly extendable if and only if 
$\alpha$ and $\beta$ are both modularly extendable. Furthermore,
if $\phi_1$ and $\phi_2$ are faithful normal states on $M$ and $N$ with associated modular extensions $\tilde{\alpha}$ and $\tilde{\beta}$, respectively, then $\tilde{\alpha}\otimes \tilde{\beta}$ is a modular extension with respect to the $\phi_1 \otimes \phi_2$.
\end{Theorem}
\begin{proof}
The proof is an immediate application of two facts. Firstly, Theorem~\ref{inv} shows
that we may consider faithful normal states instead of \fns weights, and
allows us to choose product states to prove the modular extendability 
of $M \bar{\otimes} N$. Secondly, suppose that $\phi_1$ and $\phi_2$ are faithful normal sates on $M$ and $N$, and $J_1$ and $J_2$ are their modular conjugations on $L^2(M, \phi_1)$ and $L^2(N, \phi_2)$, respectively. Then $\phi_1 \otimes \phi_2$ is a faithful normal state
on $M\bar{\otimes} N$ and its associated modular conjugation is given by
$J_1 \otimes J_2$.
\end{proof}

\begin{Corollary}
Suppose that $M$ is a properly infinite factor. Then there exist E$_0$-semigroups of type EII and type EIII on $M$. 
\end{Corollary}
\begin{proof}
Let us denote by $\CE(N)$ the set of all $E_0$-semigroups on a factor $N$, and let
$I$ denote the identity E$_0$-semigroup on $N$. 
Consider the map $\iota : \CE(\CB(\CH)) \to \CE(M \bar{\otimes} \CB(\CH))$ given by $\iota(\beta) = I \otimes \beta$. It follows from the 
previous theorem that $\iota$ takes extendable E$_0$-semigroups
to extendable E$_0$-semigroups, and furthermore the modular extension of
$\iota(\beta)$ with respect to a tensor weight is given by $ I \otimes \tilde{\beta}$, which is cocycle conjugate to $\tilde{\beta}$. It follows from Remark~\ref{rem:classification-in-BH} that $\tilde{\beta}$
has the same type as $\beta$. The result now follows from the existence of E$_0$-semigroups of type II and III  on type I$_\infty$ factors (see \cite{powers-typeII, powers-non-spatial}).
\end{proof}

\begin{Remark}
Since every modularly extendable equimodular E$_0$-semigroup $\alpha$ 
on a 
properly infinite factor $M$ has a joint unit with $\alpha'$, its modular extension
cannot be of type III. Therefore, it follows from the previous Corollary
that there exist E$_0$-semigroups which are modularly extendable yet not equimodular with respect to any weight. (Compare with Example~\ref{ex:not-equimodular}).
\end{Remark}

\subsection{Coupling Index}
The coupling index was first introduced by Margetts and Srinivasan~\cite{MS}, and to our knowledge the concept of superproduct system also appeared for the first time in \cite{MS}. 
We quickly review these definitions here for use Section~\ref{sec:examples}.

Let $(M,\phi)$ be a factorial noncommutative measure space, and 
let $\alpha$ be an  $E_0$-semigroup on  $M$. In order to simplify
notation, we will identify $M$ with $\pi_\phi(M)$, and we will
denote by $\alpha'$ the $E_0$-semigroup on $M^\prime$ obtained by
modular conjugation.

\begin{Definition}[\cite{MS}]
 A \emph{superproduct system of Hilbert spaces} is a 
 one-parameter family of separable Hilbert spaces $H=\{ (t, H_t) : t \geq 0 \}$, together
with isometries
\[
 U_{s,t} : H_s \otimes H_t \mapsto H_{s+t}, \qquad s, t \in (0, \infty)
\]
which satisfy the 
following requirements of associativity and measurability:
\begin{enumerate}
\item (Associativity) For any $s_1 , s_2 , s_3 \in (0, \infty)$
\[
U_{s_1 ,s_2 + s_3} (1_{H_{s_1}} \otimes U_{s_2 , s_3 }) = U_{s_1 +s_2 ,s_3}  (U_{s_1 ,s_2} \otimes 1_{H_{s_3}}).
\]
\item (Measurability) The space $H$ is equipped with a structure of standard Borel space that is such that the projection $p : H \to (0, \infty)$ onto the first coordinate is measurable,  and the inner
product $\langle\cdot, \cdot\rangle: \{ (\xi, \eta) \in H\times H: p(\xi)=p(\eta)\} \to \C$ is measurable.
\end{enumerate}
\end{Definition}

We will be particularly interested in the \emph{coupling superproduct system} $H_\alpha$ associated to $\alpha$ (with respect $\phi$, although 
we will suppress this dependency), defined as follows. For $t>0$ let
\begin{align*}
 E^{\alpha}(t) = \{ T \in \CB(\CH_\phi) ~: \alpha_t(x)T = Tx, ~ x \in M \},\\
 E^{\alpha^\prime}(t) = \{ T \in \CB(\CH_\phi) ~: \alpha^\prime_t(y)T = Ty, ~ y \in M^\prime \}. 
 \end{align*}
The fibers of $H_\alpha$ are given by $H_{\alpha}(t) = E^{\alpha}(t) \cap E^{\alpha^\prime}(t)$ for all $t$, and the complex-valued inner product is uniquely determined  by the identity $ y^*x= \langle x, y \rangle 1$, for  $x, y \in H_{\alpha}(t)$. We endow $H$ with relative product Borel structure on $(0,\infty)\times \CB(\CH_\phi)$, arising from
the weak*-topology on $\CB(\CH_\phi)$. See \cite{MS1} for more details,
and the proof that this is indeed a superproduct system.

\begin{Definition}
 A \emph{unit} for an $E_0$-semigroup $\alpha$ on a von Neumann algebra $M$ is a strongly continuous semigroup $\{ T_t : t \geq 0 \} $ 
 of bounded operators acting on $\CH_{\phi}$ satisfying $T_0 =I$ 
and $ \alpha_t(x) T_t  = T_tx$, for all $x \in M$.
\end{Definition}

\begin{Definition}
A  unit for a superproduct system $(H_t, U_{s,t})$ is a measurable section $\{u_t, t \geq 0 \}$ satisfying 
\[
 U_{s,t} ( u_t \otimes u_s ) = u_{s+t}, \quad s, t \in (0, \infty). 
\]
\end{Definition}

We note that the units of the coupling superproduct system $H_\alpha$
are precisely the common units for $\alpha$ and $\alpha'$. In particular,
such units may fail to exist. 

\begin{Remark}\label{rem:canonical-unit}
We note that when $\alpha$ is equimodular, units for the coupling superproduct system always exist. In fact, suppose that $\alpha$ is an equimodular $E_0$-semigroup on a factorial noncommutative measure space 
$(M, \phi)$. Then $\alpha_t$ is $\phi$-preserving for every $t$, hence  
there exists a one-parameter family of strongly continuous isometries
$\{ u_t : \geq 0 \}$ satisfying $u_t x_{\phi} = (\alpha_t(x))_\phi$, for all $x \in \FN_\phi \cap \FN_\phi^*$ and $t\geq 0$. 
Consequently we have $ u_t x  = \alpha_t(x)u_t$,  for all $x \in M$ and
$t\geq 0$, and it is clear that $(u_t)$ has the semigroup property which commutes with the modular conjugation. Hence it also interwines with
$\alpha'$ and constitutes a unit, which is called the \emph{canonical unit} for the coupling superproduct system.
 \end{Remark}

Let $\CU(\alpha, \alpha^\prime) $ be the collection of all units
of the coupling superproduct system $H_\alpha$, and suppose that
it is nonempty.
Let $S, T \in \CU(\alpha, \alpha^\prime)$ be two  units.
Then the function $f(t) = \langle S_t , T_t \rangle$ is measurable
and it satisfies  $f(t+s)= f(t)f(s)$ and $f(0)=1$. 
So there exists a complex number $c(S,T)$ such that $ \langle S_t , T_t \rangle = e^{tc(S,T)}$. The associated covariance function $c : \CU(\alpha, \alpha^\prime) \times \CU(\alpha, \alpha^\prime) \rightarrow \C$ is conditionally positive definite (by the same reasoning as in \cite[Proposition~2.5.2]{Arv}). Following the same approach as in the
definition of the Arveson-Powers index, we can define a Hilbert space
$\CH(\CU(\alpha, \alpha^\prime))$ as follows. Let 
$\C_0\CU(\alpha, \alpha^\prime)$ be the set of finitely supported zero-mean complex-valued functions endowed with the semi-definite
inner-product
$$
\langle f, g \rangle = \sum_{x,y\in \CU(\alpha, \alpha^\prime)} 
c(x,y)f(x)\overline{g(y)}.
$$
We define $\CH(\CU(\alpha, \alpha^\prime))$ to be the Hilbert space
obtained by the associated quotient and completion of $\C_0\CU(\alpha, \alpha^\prime)$ (for more details regarding this construction, see \cite[Remark~2.5.3]{Arv}). 

\begin{Definition}[\cite{MS}]
If $\alpha$ is an E$_0$-semigroup such that 
$\CU(\alpha, \alpha^\prime)$ is nonempty, then its coupling index is defined to be $\ind_c(\alpha) = \dim \CH(\CU(\alpha, \alpha^\prime))$. 
\end{Definition}

\begin{Remark}
It follows from Theorem~\ref{inv} and its proof that
the coupling superproduct systems associated to different weights
will be isomorphic via the unitary implementing the unitary equivalence
between the associated GNS representations.
By the same token, it is straightforward to apply the techniques of the proof of Theorem~\ref{prop:cocycle-conjugacy} to
show that the coupling index is a cocycle conjugacy invariant. 
\end{Remark}

\begin{Remark}\label{compute-coupling-index}
It is straightforward to check that if 
 $\alpha$ has a modular extension $\widetilde{\alpha}$, 
then the coupling superproduct system of $\alpha$ is actually
a \emph{product} system and its coupling constant is related
to the Arveson-Powers index of the modular extension by the formula
\[
\ind_c(\alpha) = \ind(\walpha)
\]
\end{Remark}

It is worth noting the perhaps inconvenient fact that when $\alpha$
is an E$_0$-semigroup of a type I$_\infty$ factor, the coupling
index turns out to be twice the Arveson-Powers index, i.e.
\[
\ind_c(\alpha) = 2 \ind(\alpha)
\]

The following proposition will be useful for the computation
of the coupling index in the examples. We omit the straightforward proof.

\begin{Proposition}\label{prop:superprod-in-space}
Let $(M,\phi)$ be a factorial noncommutative measure space,
 let $\alpha$ be an E$_0$-semigroup on $M$, and let $H_\alpha$ be its
 associated coupling superproduct system. Suppose that $\Omega \in \CH_{\phi}$ is a normalized vector. Then the map $\rho: H_\alpha \to 
 (0,\infty) \times \CH_\phi$ given by
 \[
 \rho(t, T) = (t, T\Omega)
 \]
 is injective, isometric fiberwise and measurable when the
 range has the canonical Borel product structure. In particular
 the $\rho(H_\alpha)$ has a natural superproduct system structure 
 via the pushforward, under which it is isomorphic to $H_\alpha$.
 \end{Proposition}

\subsection{Relative commutant index}
In this subsection we introduce an invariant for certain
$E_0$-semigroups on a factorial noncommutative measure space $(M, \phi)$,
which is a generalization of the invariant defined by Alevras~\cite{Alev} for the context of II$_1$ factors. 

Let us quickly review Hideki Kosaki's notion of index
for a subfactor of a general factor (see \cite{kosaki}).  
Let $N$ be a subfactor of a factor $M$ and let $E: M \rightarrow N$ be a faithful normal conditional expectation. Haagerup~\cite{Haa1, Haa2}) proved that there exists a faithful normal operator-valued weight 
$E^{-1}: N^\prime \rightarrow M^\prime$ which is characterized by
the following identity: if $\phi$ is an \fns weight on $N$ and
$\psi$ is an \fns weight on $M'$,
\[
\frac{d(\phi \circ E)}{d(\psi)} = \frac{d(\phi)}{d(\psi \circ E^{-1})},
\]
where 
$d(\phi \circ E)/{d(\psi)}$ and $ d(\phi)/{d(\psi \circ E^{-1})}$ are Connes spatial derivatives 
(see \cite{Con}). The Kosaki index of $E$, which is a scalar, is defined by
\[
(\ind E)\, 1 = E^{-1}(1)
\]
Let $\CE(M, N)$ be the collection of all faithful
normal conditional expectations from $M$ onto $N$. Then the
\emph{minimal index of the pair $N\subseteq M$} is defined to be
$$
[M: N] = \min \{ \ind E \,:\,  E \in  \CE(M,N) \}.
$$
We note that if $\ind E = \infty$ for some $E \in \CE(M,N)$, then
it is infinite for all elements of $\CE(M,N)$, in which case
$[M:N] = \infty$ (see \cite{FH}). In fact, there 
exists $E_0 \in \CE(M,N)$ such that 
$[M:N] = \ind E_0$. We note that if $\gamma$ is an automorphism of $M$,
then by \cite[Theorem~2.2]{kosaki}, 
\begin{equation}\label{eq:index}
[M  : \gamma(N) ] = [M : N]
\end{equation}

\begin{Definition}
Let $M$ be a factor, and let $\alpha = \{ \alpha_t : t \geq 0 \} $ be 
an $E_0$-semigroup  on $M$. For every $t \geq 0$, let
$N_{\alpha}(t) = ( \alpha_t(M)^\prime \cap M) \vee \alpha_t(M)$
be the von Neumann algebra generated by
$\alpha_t(M)^\prime \cap M $ and $\alpha_t(M)$. 
We denote by $\CI_{\alpha}$ the set of all $t \geq 0$ such that $N_{\alpha}(t)$ is a subfactor of $M$
and $\CE(M, N_t) \neq \varnothing$. For every $t \in \CI_\alpha$,
let 
\[
c_{\alpha}(t) = [ M : N_{\alpha}(t) ]. 
\]
If $\CI_\alpha \neq\varnothing$, then we define the \emph{relative commutant index of $\alpha$} to be the family $(c_\alpha(t))_{t\in \CI_\alpha}$.
\end{Definition}

\begin{Lemma}\label{C-cond}
Let $(M,\phi)$ be a noncommutative probability space, and let $\alpha$ be an equimodular $E_0$-semigroup on $M$. Then for every $t \geq 0 $, there exists a faithful normal conditional 
expectation $ E_t : M \rightarrow N_\alpha(t)$. 
\end{Lemma}
\begin{proof}
By Theorem~\ref{equi}, we have that $\alpha_t(M)$ is invariant under the modular automorphism group $\sigma_t^\phi$, hence so is its relative commutant. Therefore, $N_\alpha(t)$  is also invariant under the modular
group. It follows from Takesaki's theorem, that there exists a $\phi$-preserving faithful normal conditional expectation
$E_t$ from $M$ onto $N_t$. 
\end{proof}

\begin{Remark}
It follows that if $\alpha$ is equimodular, then $t \in \CI_\alpha$ 
if and only if $N_\alpha(t)$ is a factor. Thus, by \cite[Corollary~10.7]{SS} we have  that $t \in \CI_\alpha$  if and
only if $\alpha_t(M)'\cap M$ is a factor. In all the examples
we consider in this paper, this condition holds for all $t\geq 0$.
\end{Remark}

\begin{Proposition}\label{prop:relative-commutant-index}
Let $M$ and $N$ be factors and let 
 $\alpha$ and $\beta$ be $E_0$-semigroups on $M$ and $N$, respectively.
 Suppose that  $\CI_\alpha$ and  $\CI_\beta$ are nonempty. Then

\begin{enumerate}[ label={(\roman*)}, start=1, leftmargin=1cm , itemsep=0.2cm]
 \item The relative commutant index of $\alpha$, that is
 the family $\{c_{\alpha}(t)\}_{t \in \CI_\alpha}$, 
 is invariant under conjugacy and  cocycle conjugacy.
 \item For all   $t \in \CI_\alpha \cap \CI_\beta$ we have that
$
 c_{\alpha \otimes \beta}(t) = c_{\alpha}(t)\cdot c_{\beta}(t).
$
  \end{enumerate}
\end{Proposition}
 
\begin{proof}
\begin{enumerate}[ label={(\roman*)}, start=1, leftmargin=1cm , itemsep=0.2cm]
 \item 
 Since the minimal index is invariant under automorphisms cf.\ eq.~\eqref{eq:index},  it
 is conjugacy invariant. We now prove that it is an invariant of
 cocycle equivalence.
Let $\{ U_t : t \geq 0 \}$ be an $\alpha$-cocycle in $M$ and suppose
that $\beta_t = \text{Ad}(U_t)\circ \alpha_t $.  It is
 straightforward to see that 
\[
 ({\beta_t(M)}^\prime \cap M ) \vee \beta_t(M) = U_t( ({\alpha_t(M)}^\prime \cap M ) \vee \alpha_t(M)) U_t^{*},
\]
i.e.\ $N_{\beta}(t)  =U_t N_{\alpha}(t) U_t^* $. So by eq.~\eqref{eq:index}, we have that
$[M : N_{\alpha}(t) ]= [M : N_{\beta}(t) ]$, i.e.\ $c_{\alpha}(t) = c_{\beta}(t)$.
\item
We have, for all $t \geq 0$, 
\begin{align*}
N_{\alpha\otimes \beta}(t) 
& =\left({ {(\alpha_t\otimes \beta_t)(M \otimes N)}^\prime \cap  (M\otimes N) } \right) 
\vee (\alpha_t\otimes \beta_t)(M\otimes N)\\
&= \left( {({\alpha_t(M) \otimes \beta_t(N)})^\prime \cap M\otimes N }\right) 
\vee \alpha_t(M) \otimes \beta_t(N)  \\
&= \left(({\alpha_t(M)}^\prime \otimes {\beta_t(N)}^\prime) \cap M\otimes N \right) 
\vee  \alpha_t(M) \otimes \beta_t(N) \\
&= \left(({\alpha_t(M)}^\prime \cap M) \otimes ({\beta_t(N)}^\prime \cap N) \right) 
\vee  \alpha_t(M) \otimes \beta_t(N) \\ 
&= \left(({\alpha_t(M)}^\prime \cap M) \vee  \alpha_t(M)\right) \otimes \left(({\beta_t(N)}^\prime \cap N) 
\vee \beta_t(N)\right) = N_\alpha(t) \otimes N_\beta(t)
\end{align*}
Then the multiplicative property of the minimal index over the tensor product completes the proof (see \cite[Corollary~5.6]{longo}). 
\end{enumerate}
\end{proof}

\begin{Remark}\label{relative-index-extendable}
Let $(M,\phi)$ be a noncommutative measure space, and let $\alpha$
be an equimodular E$_0$-semigroup on $M$. By
Theorem~\ref{E0-extcon}, if $I_\alpha = [0, \infty)$
and $c_\alpha(t)=1$ for all $t$, we have that $\alpha$ is 
modularly extendable. Conversely, also by 
Theorem~\ref{E0-extcon}, when $\phi$ is a faithful state and 
$\alpha$ is modularly extendable
we have that
$N_\alpha(t) = M$ for every $t$. Therefore $I_\alpha = [0, \infty)$
and $c_\alpha(t)=1$ for all $t$.
\end{Remark}

\section{Examples}\label{sec:examples}

In this section we determine the modular extendability, coupling index and relative commutant index for the following examples of E$_0$-semigroups:  $q$-CCR flows for $q\in (-1,1)$ and CAR flows.

For the sake of comparison, we start by commenting briefly on the CCR flows, which have been studied extensively by Margetts and Srinivasan in \cite{MS1}.

For $q\in (-1,1)$, the $q$-CCR flows provide an interesting generalization of the CCR-flows (which would correspond to the case $q=1$). The $q$-CCR flows however, in contrast to the CCR flows, 
 turn out \emph{not} to be modularly extendable. As we will discuss, the $q$-CCR flows act on the so called $q$-Gaussian II$_1$ factors, which are not injective,  do not have property $\Gamma$ and are strongly solid. 

The CAR flows provide similar examples of E$_0$-semigroups which are equimodular but not modularly extendable, and they act on hyperfinite factors of type II$_1$, type II$_\infty$ and type III$_\lambda$ for $\lambda \in (0,1)$,
depending on the choice of quasi-free state.

To our knowledge, the $q$-CCR flows have not been considered directly in the literature earlier from the point of view of classification of E$_0$-semigroups. The CAR flows for a subset of 
quasi-free states considered here, appeared earlier in \cite{BISS} and
\cite{Bikram}, respectively. In this paper we are interested in their
invariants, which have not been computed before.

For the remainder of this section, let $\CK$ be a separable Hilbert
space and let $\CH = L^2 (0, \infty ) \otimes \CK$ and let 
$\{S_t\}_{t \geq 0 }$ be the shift semigroup on $\CH$ defined by
\[
(S_tf)(s) = \begin{cases}   0, & \quad s<t,\\
  f(s-t),&  \quad s \geq t.
\end{cases} 
\]

\subsection{CCR flows} We review the definition of the CCR flows. Let 
$\Gamma_s(\CH)$ denote the symmetric or Bosonic Fock space
with one-particle space $\CH$ and vaccum vector $\Omega$. 
Given $f \in \CH$, let $W(f) \in \CB(\Gamma_s(\CH))$ be the Weyl operator uniquely determined by
\[
W(f) \Omega = \exp(f) := \sum_{n=0}^\infty \frac{1}{\sqrt{n!}} \, f^{\otimes^n}
\]
The CCR algebra $\ccr(\CH)$ is the C*-algebra generated by the 
all Weyl operators over $\CH$.

Let $A\in \CB(\CH)$ be an operator such that  $A-1$ is positive. 
There exists a unique state $\varphi_A$ on $\ccr(\CH)$, called the quasifree state with symbol $A$, satisfying
\[
\varphi_A(W(f)) = e^{-\frac{1}{2} \langle Af, f \rangle}
\]
In addition, when $A-1$ is injective, the GNS representation of
$\varphi_A$ can be described explicitly as follows. Let $T = \frac{1}{2}(A-1)$ and let $q$ be an anti-unitary on
$\CH$ such that $qS_t = S_tq,$ for $t\geq 0$.  
Let $\pi_A$ be the representation of $\ccr(\CH)$
on  $\Gamma_s(\CH)\otimes \Gamma_s(\CH)$ satisfying
\[
 \pi_A(W(f) ) = W
(\sqrt{1+T}f)\otimes W(q\sqrt{T}f ), \qquad \forall f \in \CH
\]
Then it is straightforward to check that 
$\pi$ is the GNS representation for $\varphi_A$ with 
cyclic and separating vector $\Omega \otimes \Omega$.

\begin{Definition}
Let $\CK$ be a Hilbert space, let $\CH = L^2(0,\infty)\otimes \CK$,
and let $A \in \CB(\CH)$ be an operator such that $A\geq 1$,  $A-1$ is 
injective and $S_t^*AS_t = A$ for all $t$. 
The \emph{CCR flow} corresponding to $A$ is the unique  $E_0$-semigroup $\beta^{A}$ on $M_A = \pi_A(\ccr(\CH))''$ satisfying 
\[
\beta^{A}_t(\pi_A(W(f)))=\pi_A(W(S_tf)), \qquad  \forall f \in \CH, t\geq 0.
\]
\end{Definition}

For simplicity, let us fix such a Hilbert space $\CK$ and operator
$A \in \CB(\CH)$ such that $A\geq 1$, $A-1$ is injective and 
and $S_t^*AS_t = A$ for all $t$. The existence of the CCR flow $\beta^A$ is a direct consequence of a straightforward generalization of \cite[Proposition 2.1.3]{Arv}.

Margetts and Stinivasan~\cite[Proposition 7.5]{MS1} proved that $\beta^A$ is 
equimodular if and only if there exists $R \in \CB(\CK)$ such
that $A = 1 \otimes R$, and in that case $\beta^A$ is modularly
extendable with modular extension given by the CCR flow on $\CB(\Gamma_s(\CH)\otimes\Gamma_s(\CH) )$ of index $2\kappa$. 
In summary, in our language \cite[Proposition 7.5]{MS1} states that the CCR flow associated 
to  $A = 1 \otimes R$ is equimodular, modularly extendable, has type EI and $\ind_c(\beta^A) = 2 \kappa$.

Moreover, by Remark~\ref{relative-index-extendable}, since in this case $\beta^A$ is equimodular and modularly extendable, its relative commutant index
is the constant family equal to 1.

We note that if $A = 1 \otimes R$ for $R\geq 1$, and $A-1$ injective,
we have that $M_A$ is a type III factor (see \cite{MS1}). For example, for $\lambda \in (0,1)$, let $A=\frac{1+\lambda}{1-\lambda}$.
Then $A\geq 1$ and $A-1$ is invertible,  and  $M_\lambda= \pi_\lambda(\ccr(\CH))''$ is a type III$_\lambda$
factor (see \cite{Hol,ArWo}).

\subsection{$q$-CCR flows}

Here we discuss examples of $E_0$-semigroups arising from the $q$-cano\-nical commutation relations. 
For more details on the basic construction see \cite{BS,BKS}.
Following the convention used in the literature for compatibility
of the formulas, in this subsection
our inner product will be conjugate linear in the first entry.

Let $\CK_\R$ be a \emph{real} Hilbert space and let  
$\CK = \CK_\R + i \CK_\R$ be its complexification. Let
$\CH_\R = L^2 (0, \infty; \CK_\R)$ be the real Hilbert
space of square integrable functions taking values in $\CK_\R$, and
let $\CH= L^2 (0, \infty ;\CK) $, which is the complexification of $\CH_\R$.   

Let $ q \in (-1, 1) $ is a fixed real number. 
Let  $\CF_f(\CH)$ be the linear span of vectors of the form 
 $f_1 \otimes f_2 \otimes \cdots \otimes f_n \in {\CH}^{\otimes n } $ (with varying
$n \in \N$), where we set ${\CH}^{\otimes 0 } \cong \C \Omega$  for some distinguished vector,  
called the vacuum vector. On $\CF_f(\CH)$, we consider the sesquilinear form $\langle \cdot, \cdot \rangle_q$ 
given by the sesquilinear extension of 
\[
 \langle f_1 \otimes f_2 \otimes \cdots f_n  ,~ g_1 \otimes g_2 \otimes \cdots g_m \rangle_q
:= \delta_{mn}\sum_{ \pi \in S_n }q^{i(\pi)} 
\langle f_1,  g_{\pi(1)}\rangle \cdots \langle f_n,  g_{\pi(n)}\rangle 
\]
where $S_n$ denotes the symmetric group of permutations of $n$ elements and $i(\pi)$ is the
number of inversions of the permutation $\pi \in S_n$, defined by
\[
 i(\pi) := \# \big\{(i, j) \,\big|\, 1 \leq i < j \leq n, \;
 \pi(i) > \pi(j)\big\}.
\]
 The $q$-Fock space $\CF_q(\CH)$ is the completion of 
 $\CF_f(\CH)$ with respect
to $\langle \cdot, \cdot \rangle_q$.  
Given $f \in \CH$, the creation operator $l(f)$ on $\CF_q(\CH)$ 
is the bounded operator defined by
\begin{align*}
l(f )\Omega  & = f,  \\
l(f )f_1 \otimes \cdots \otimes f_n & = f \otimes f_1 \otimes \cdots\otimes f_n,
\end{align*}
and its adjoint is the annihilation operator $l(f)^*$ given by 
\begin{align*}
 {l(f )}^*\Omega & = 0,\\
{l(f )}^* f_1 \otimes \cdots \otimes f_n 
& = \sum^{n}_{i=1}q^{i-1} \langle f,  f_i \rangle f_1 \otimes \cdots\otimes \breve{f_i}\otimes \cdots \otimes f_n,  
\end{align*}
We have that the following $q$-canonical commutation relation is 
satisfied: 
\[
l(f)^* l(g) - ql(g)l(f)^* = \langle f, g \rangle \cdot1 \qquad f, g \in \CH.
\]
For $f \in \CH_\R $, we define the self-adjoint operator $W(f) =  l(f ) + {l(f)}^*$, and we define the von Neumann algebra 
\[ 
\Gamma_q(\CH_\R) = \{ W(f) \,|\,  f \in \CH_\R \}''. 
\]

We recall that for every $q \in (-1,1)$, the so called $q$-Gaussian von Neumann algebra $\Gamma_q(\CH_\R)$ is a II$_1$ factor (see \cite{BKS}), which is not injective (see \cite{nou}),  does not have property $\Gamma$ (see \cite{sniady}) and it is strongly solid (see \cite{avsec}). The vector state $\tau(x)= \langle x \Omega, \Omega \rangle_q $ is the trace for $\Gamma_q(\CH_\R)$, hence 
$\Gamma_q(\CH_\R)$ is in standard form in $\CB(\CF_q(\CH))$ with respect
to the the cyclic and separating vector $\Omega$. Therefore we have
a well-defined injective map $W: \Gamma_q(\CH_\R)\Omega \to \Gamma_q(\CH_\R)$ uniquely determined by the identity $W(\xi)\Omega = \xi$ for $\xi \in \Gamma_q(\CH_\R)\Omega$ (we note that this definition
of $W$ extends the previous one since $W(f)\Omega = f$ when $f\in \CH_\R$). 
Let $e \in \CH_\R$ be a vector of norm one and denote by $E_e$ the closed subspace of $\CF_q (\CH)$ spanned
by the elements $\{e^{\otimes n} \,|\, n \geq 0\}$, i.e.\ $ E_e = \CF_q(\C e)$. It is straightforward to check that for $\xi \in \Gamma_q(\CH_\R)\Omega$, we have that
 $W(\xi) \in W(e)''$
if and only if $\xi \in E_e \cap \Gamma_q(\CH_\R)\Omega$. 
 
\begin{Definition}
Suppose that $q \in (-1,1)$.
Let $\{S_t\}_{t\geq 0 }$ denote the shift semigroup on $\CH$, and 
also its restriction to $\CH_\R$. The \emph{$q$-CCR flow} of rank $\dim \CK_\R$ is the unique E$_0$-semigroup
$\alpha^q$ on $\Gamma_q(\CH_\R)$ such that
 \[
  \alpha^q_t(W(f)) = W(S_tf), \quad   f \in \CH_\R. 
 \]
 \end{Definition}

We note that the $q$-CCR flow is a well-defined E$_0$-semigroup, 
since it is obtained
via the second quantization functor $\Gamma_q$ introduced by 
Bo\.zejko, K\"ummerer and Speicher~\cite{BKS}.

\begin{Theorem}
Suppose that $q\in (-1,1)$ and let $\alpha^q$ be q-CCR flow corresponding to a real Hilbert space $\CK_\R$.
\begin{enumerate}
\item The $q$-CCR flow $\alpha^q$ is equimodular with respect
to the trace $\tau(x) = \langle x\Omega, \Omega\rangle_q$.
\item $\Gamma_q(\CH_\R) \cap {\alpha^q_t (\Gamma_q(\CH_\R))}^\prime = \C\cdot 1 $ for all $t \geq 0$. 

\item the $q$-CCR flow $\alpha^q$ is not modularly extendable.

\item Let $\CF_q(S_t) \in \CB(\CF_q(\CH))$ be the quantization of the shift
$S_t$. The coupling su\-per\-pro\-duct system $(H_{\alpha^q}(t))_{t\geq 0}$ is given by $H_{\alpha^q}(t) = \C \cdot \CF_q(S_t)$ and  multiplication is given by operator multiplication.

\item The coupling index of $\alpha^q$ is zero.

\item The relative commutant index of $\alpha^q$ is the constant family equal to $\infty$.
\end{enumerate}

\end{Theorem}
\begin{proof}
(1)  Any unital normal $*$-endomorphism on a II$_1$ factor is equimodular with respect to the trace  (see \cite{BISS}), hence
$\alpha^q$ is equimodular with respect to $\tau$.

(2) Let $t\geq 0$ be fixed, and let $x \in \Gamma_q(\CH_\R) \cap {\alpha^q_t (\Gamma_q(\CH_\R))}^\prime$.  There exists 
$\xi \in \Gamma_q(\CH_\R)\Omega$ such that $x = W(\xi)$. 
Let $f \in \CH_\R$ be a vector of norm one and let $e=S_tf$.
Notice that $[x, W(e)] = [x, \alpha_t^q(W(f))] = 0$. By 
\cite[Theorem 1]{ER}, $W(e)''$ is a maximal abelian subalgebra 
of $\Gamma_q(\CH_\R)$, hence we must have that $x \in W(S_tf)''$.
Therefore  $\xi \in E_{S_tf}$,  for every vector $f \in \CH_R$ of norm one. Thus $x \in \C\cdot 1$.

(3) It follows from item (2) that 
$(\alpha^q_t(\Gamma_q(\CH_\R)) \cup ( \Gamma_q(\CH_\R) \cap \alpha^q_t (\Gamma_q(\CH_\R))') '' = \alpha^q_t(\Gamma_q(\CH_\R)) \neq \Gamma_q(\CH_\R)$. Therefore by item (1) and Remark~\ref{stateext},
we have that $\alpha^q$ is not modularly extendable.

(4) By \cite[Proposition 8.11]{MS} that 
$H_{\alpha^q}(t)\Omega$ is the closure of $(\Gamma_q(\CH_\R) \cap {\alpha^q_t (\Gamma_q(\CH_\R))}^\prime ) \Omega$. Therefore, by
Proposition~\ref{prop:superprod-in-space}, we have that $H_{\alpha^q}(t)$
is one-dimensional for every $t>0$.  Notice that $\CF_q(S_t)$ is  unit
of $\alpha^q$ and by equimodularity it is also a unit of $(\alpha^q)'$.
Hence we obtain that $H_{\alpha^q}(t) = \C \cdot \CF_q(S_t)$. It
is clear that the multiplication is given by  operator multiplication.

(5)  and (6) follow trivially from the previous items.
\end{proof}

\begin{Remark}
We note that the $q$-CCR flow provides an example of 
equimodular E$_0$-semigroup whose superproduct system is actually
a product system despite the fact that it is not modularly extendable.
\end{Remark}

\subsection{CAR flows}\label{CAR} 
Let $\CK$ be a Hilbert space and let $\CH = L^2(0, \infty)\otimes \CK$.
Let $\CF_-(\CH)$ denote the anti-symmetric Fock space
with vacuum vector $\Omega$.
For  
$f \in \CH$, let $c(f) \in \CB(\CF_-(\CH))$ be the 
creation operator given by
\[
c(f)\Omega = f, \qquad\quad
c(f) f_1 \wedge \cdots \wedge f_n = f \wedge f_1 \wedge \cdots \wedge f_n,
\qquad f_1, \dots, f_n \in \CH
\]
We note that the map $\CH \to \CB(\CF_-(\CH))$ given by
$f \mapsto c(f)$ is $\C$-linear, and it satisfies
the canonical commutation relations
\[
c(f)c(g) + c(g)c(f) = 0 \quad  \text{ and } \quad  c(f)c(g)^* + c(g)^*c(f) = \langle f, g\rangle 1, \qquad f, g \in \CH.
\]
where of course $1$ denotes the identity operator.
The CAR algebra $\CA(\CH)$ is the unital $C^*$-algebra generated
by  $\{a(f): f \in \CH\}$ in $\CB(\CF_- (\CH))$. We note that 
$||a(f)|| = ||f||$ for $f \in \CH$.
 Now suppose $R \in \CB(\CH)$ satisfies $0 \leq R \leq 1$.
Every such operator $R$ determines a unique state $\omega_R$ on $\CA(\CH)$,
called the quasi-free state with two-point function $R$,
which satisfies the following condition:
\[
\omega_R (c^*(f_m ) \cdots c^*(f_1)c(g_1 ) \cdots c(g_n )) = \delta_{mn} \det (\langle  g_i, R f_j \rangle ).
\]
We will also use the definition of the even CAR algebra. 
Let $\gamma$ be the unique unital automorphism of $\CA(\CH)$ such that
$\gamma(c(f)) = -c(f)$ for all $f\in \CH$. The even CAR algebra is
the subalgebra $\CA_e(\CH) = \{x\in \CA(\CH)  \;|\; \gamma(x) = x \}$. 
It is easy to show that the even CAR algebra is generated as a C*-algebra by the homogeneous monomials of even degree on creation
and annihilation operators.

\begin{Definition}
  Suppose that $R \in \CB(\CH)$ satisfies $0\leq R \leq 1$ and $S_t^*RS_t = R$ for all $t\geq 0$, and let $\pi_R$ be the GNS
  representation for $\omega_R$. Then the unique E$_0$-semigroup $\alpha^R$ on $M_R = \pi_R(\CA(\CH))''$  satisfying
 \[
  \alpha^R_t( \pi_R(c(f)) = \pi_R(c(S_tf)), \qquad f \in \CH, t\geq 0
 \]
 is called the \emph{CAR flow of  rank dim $\CK$} (on $M_R$) associated to the operator $R$.  
 \end{Definition}
 
 It follows from a straightforward generalization of \cite[Proposition 13.2.3]{Arv} for the context of factors of all possible types, that the CAR flow associated to $R$ is well-defined. We also note that $M_R$ is
 always a hyperfinite factor (see \cite{PoSt}).
 
In the case that $0\leq R \leq 1$ in $\CB(\CH)$ satisfies the additional
conditions that $R$ and $1-R$ are invertible, we have a convenient 
description of the GNS representation of $\omega_R$. Indeed, let $Q$ be an anti-unitary operator on $\CH$ with $Q^2=1$, and let
$\Gamma$ be the unique unitary operator on $\CF_-(\CH)$
such that $\Gamma \Omega = \Omega $ and
$\Gamma c(f) = -c(f)\Gamma$, for all $f \in \CH$.  
Then there exists a representation 
$\pi_R$ of the $C^*$-algebra $\CA(\CH)$ on the Hilbert space
$\CH_R = \CF_-(\CH) \otimes \CF_-(\CH)$ defined by the 
following formulas (see  for instance \cite{BrRo}):  for all 
$f \in \CH$,
\begin{align*}
\pi _R(1) & = 1, \\
 \pi_R (c(f)) & = c((1-R)^{1/2}f) \otimes \Gamma +  1 \otimes c^* (Q R^{1/2}f), \\
\pi_R (c^*(f)) &= c^*((1-R)^{1/2}f) \otimes \Gamma +  1 \otimes c (Q R^{1/2}f).
\end{align*}
When $R$ and $1-R$ are invertible, the representation $\pi_R$ on $\CH_R$ is the
GNS representation for $\omega_R$ with respect to the cyclic vector
$\Omega \otimes \Omega$. We denote the normal extension of the quasifree state $\omega_R$ to $M_R$ by the same symbol $\omega_R$.   We will often write $c(f)$ instead of $\pi_R(c(f))$ to lighten notation when the representation is determined by the context.

 \begin{Lemma}\label{cond-exp-car}
 Let $R \in \CB(\CH)$ be an operator such that $0\leq R \leq 1$,  $S_t^*RS_t = R$ for all $t\geq 0$, and $R$ and $1-R$ are
 invertible. Suppose that $\CK$ is a closed subspace of $\CH$
 such that $R \CK \subseteq \CK$. Then there exists a unique normal $\omega_R$-preserving conditional expectation of $M_R$ onto
 $\pi_R(\CA(\CK))''$.
 \end{Lemma}
 \begin{proof}
  Let $\{ \sigma_t\}_{t \in \R }$ be the 
  modular automorphism group on $M_R$ with respect to the normal state $\omega_R$. 
  Then the KMS condition for the modular automorphism group implies that (see \cite[Example 5.3.2]{BrRo}),
\[ 
\sigma_t(c(f)) = c(R^{it} (1-R)^{-it}f), \qquad f \in \CH. 
\]
Since $R\CK \subseteq \CK$, we have that $\sigma_t(\pi_R(\CA(\CK))) \subseteq \pi_R(\CA(\CK))$. Now it follows from 
Takesaki's theorem (see \cite[Section 3, p.309]{Ta}) that there exists a unique  normal $\omega_R$-preserving conditional
expectation from $M_R$ onto $\pi_R(\CA(\CK)))''$.
 \end{proof}

The next three lemmas are certainly known to the experts, 
however we did not find a direct reference in the literature. Hence we provide their proofs here for the convenience of the reader. The 
authors thank M. Izumi for pointing out a slick proof for Lemma~\ref{outer}. 
Lemma~\ref{rel} generalizes a result in \cite{Bikram}, and the proof below uses a
crossed product idea suggested by an anonymous referee of that paper.

\begin{Lemma}\label{outer}
Let $\CK$ be a Hilbert space and $\gamma$ be the period two automorphism 
of 
 $\CA(\CK)$ given by $\gamma (c(f)) = -c(f)$ for $f \in \CK$.
Let $R \in \CB(\CK)$ be a positive contraction, let $\omega_R $ be the quasi-free state 
of $\CK$ of $R$, and let $\pi_R$ be the GNS representation for $\omega_R$. Since $\gamma$
preserves $\omega_R$, 
it extends to an automorphism $\gamma_R$ on the weak closure $M_R = \pi_R (\CA(\CK))''$.
Then the automorphism $\gamma_R$ is inner if and only if $\Tr(R - R^2) < \infty $. 
\end{Lemma}

\begin{proof}
If $\Tr(R - R^2) < \infty $ the von Neumann algebra $M_R$ is a type I factor by \cite[Lemma 5.3]{PoSt}, hence
every automorphism of $M_R$ is inner.

Conversely, suppose that $\gamma_R$ is inner.  Then there exists a  unitary $u \in M_R$ such that $u^2=1$
satisfying $\gamma_R = \text{Ad}(u)$. Note that $\gamma_R(u) = u$
and therefore it is even in the sense that $u \in \pi_R(\CA_e(\CK))''$.
We consider the purification of 
 $\omega_R$ as in \cite{PoSt}. Let 
\[
 E_R =
 \begin{pmatrix}
  R & \sqrt{R(1-R)} \\
  \sqrt{R(1-R)} & 1-R 
 \end{pmatrix}
\qquad \text{and} \qquad
 p = 
  \begin{pmatrix}
  1 & 0 \\
  0 & 0
 \end{pmatrix}
 \]
which are projections in $\CB(\CK^2)$ where $\CK^2 =  \CK \oplus \CK$. Then we have that $M_R$
is isomorphic to $\pi_{E_R}(\CA(p\CK^2))''$. Therefore, using this isomorphism, there exists $u \in \pi_{E_R}(\CA(p\CK^2))''$ such that
$u^2=1$ and $uc(f)u = -c(f)$ for every $f \in p\CK^2$, and furthermore
$u$ is even in the sense that $u \in \pi_{E_R}(\CA_e(p\CK^2))''$. 
Let $\Ad(u)$ be the associated automorphism of $\pi_{E_R}(\CA(\CK^2))''$. 
Since $u$ is a even, i.e.\ $u \in \pi_{E_R}(\CA_e(p\CK^2))''$, 
we have that, 
\[
\Ad(u)(c(f)) = \begin{cases}   -c(f), & \quad f \in p\CK^2,\\
  c(f), &  \quad f \in (1-p) \CK^2,
\end{cases} 
\]
Hence it is  straightforward to check that $\Ad(u) \circ \pi_{E_R}$ can be identified with $\pi_{F}$, where 
\[
 F =
 \begin{pmatrix}
  R & -\sqrt{R(1-R)} \\
  -\sqrt{R(1-R)} & 1-R 
 \end{pmatrix}
 \]
Since $\pi_{E_R}$ and $\pi_{F}$ are unitarily equivalent, by \cite[Theorem 2.8]{PoSt}
we have that $E_R - F$ is 
Hilbert-Schmidt, that is $\Tr( R- R^2) < \infty$.  
\end{proof}

\begin{Lemma}\label{relative}
Let $\FH$ be a Hilbert space and 
let $R \in \CB(\FH)$ be a positive contraction. Let $\omega_R$ be the quasi-free state associated to $R$, and let $\pi_R$ be its GNS representation.
Let $B_e = \pi_R(\CA_e( \FH))''$ and let $B = \pi_R(\CA( \FH))''$. 
Then $B_e$ is a factor, and 
we have in addition that  $\Tr(R - R^2) =\infty$, then
$B_e' \cap B = \C 1$.
\end{Lemma}
\begin{proof}
 Let $\gamma \in \mathrm{Aut}(B)$ be given by
$\gamma(\pi_R(c(g)))=-\pi_R(c(g))$ for $g\in K$. Note that $B$ is
a factor and $B_e$ is the fixed point algebra of $B$ under $\gamma$,
which has period two, therefore $B_e$ is a factor. 
 Let $f \in \FH$ with $|| f|| = 1$, consider
$ u = \pi_R(c(f)) + \pi_R(c(f))^*$ and notice that  
$\gamma(u)=-u$ and $u^2=1$. Let $ \sigma = \Ad(u)$ on $B_e$. It is
straightforward to check that $B$ is isomorphic
to the crossed product $B_e \rtimes_\sigma \Z/{2\Z}$.  Moreover,  $\gamma$ implements the dual action of $\sigma$ on $B_e$ via this isomorphism.
When $\Tr( R- R^2) = \infty$,  by Lemma \ref{outer} we have that $\gamma$ is outer, hence $\sigma$ is also outer on $B_e$. 
Since $B_e$ is a factor, it follows that $\sigma$ acts freely on $B_e$. 
Every element of $x \in B$ can be written uniquely as $y + z u$ for
$y, z \in B_e$. Hence a straightforward computation shows that 
$B_e' \cap B = \C I$. 
\end{proof}

\begin{Lemma}\label{rel}
Let $R \in \CB(\CH)$ be an operator such that $0\leq R \leq 1$   
and $R$ and $1-R$ are invertible. Let $\CK$ be a closed subspace
of $\CH$ such that $R\CK \subseteq \CK$, and suppose that
$\mathop{Tr}(R|_\CK - R|^2_\CK ) = \infty$. Let 
$\CA_e(\CK^\perp)$ be the even part of $\CA(\CK^\perp)$.
 Then
\[
M_R \cap \pi_R(\CA(\CK))' = \pi_R(\CA_e(\CK^\perp))''.
\]
\end{Lemma}

\begin{proof}
We may assume that $\CK \neq \CH$. Let $N= \pi_R(\CA(\CK))''$ and let
$P = \pi_R (\CA_e(\CK^\perp) )''$. It is clear that $N$ and $P$ commute. By Lemma~\ref{cond-exp-car}, there
exists a conditional expectation $E$ from $M_R$ onto $N$ which is
normal and $\omega_R$-preserving, and it is faithful since $\omega_R$
is faithful when $R$ and $1-R$ are invertible.  It follows from 
\cite[Theorem 9.12, p.124]{SS} that $(N\cup P)''$ is isomorphic to $N \bar{\otimes} P$. 

Notice that $N$ and $P$ are canonically identified with their cutdowns
by the projection onto the closure of $\pi_R(\CA(\FH))\Omega \otimes \Omega$  when $\FH=\CK$ and $\FH=\CK^\perp$, respectively.
Hence it
 follows from Lemma~\ref{relative} that
 $N$ and $P$ are subfactors of $M_R$. 
Let $f \in \CK^\perp$ be a fixed vector with $\|f\| = 1$, and set $u = \pi_R (c(f ) + c^*(f ))$. Then $u$ is a self-adjoint unitary, so
$u^2=1$, and it normalizes $N$ and $P$. We denote by $\gamma$ the restriction of $\Ad(u)$ to $(N\cup P)'' \cong N \bar{\otimes} P$,
and let $\gamma_N$  and $\gamma_P$ be its restrictions to $N$ and $P$,
respectively. Note that 
\[
\gamma_N(\pi_R (c(f))) = -\pi_R (c(f)), \qquad f \in  \CK. 
\]
Since $\Tr(R|_\CK - R|^2_\CK ) = \infty$, by Lemma~\ref{outer}, we have that $\gamma_N$ is outer (relative to $N$). Hence $\gamma \cong \gamma_N \otimes \gamma_P$
must be outer (recall that the tensor product of automorphisms is inner if and only if both automorphisms are inner).

Since ${\pi_R (A(\CK^\perp) )}^{\prime \prime} = P + P u$, we see that $M_R$ is generated by $(P \cup N)''$ and
$u$. Moreover, notice that $M_R$ is isomorphic to the crossed product
$(N \bar{\otimes} P)\rtimes_\gamma \Z/2\Z$. 
In particular, every $x \in M_R$ is uniquely expressed as $x = y + zu$ with $y, z \in N \bar{\otimes} P$ . It remains to show that 
$M \cap N^\prime  = P$. Let $ x = y + zu$ with $y, z \in N \bar{\otimes} P$, and suppose that 
$xa = ax $ for all $a \in N$. Then we have that $ya = ay$ and $z\gamma(a) = a z$, for all $a \in N$. 
Hence $y \in P $ and  $z \in (\pi_R(\CA_e( \CK))' \cap N ) \otimes P$. By Lemma~\ref{relative} for $\FH=\CK$, 
we have that $\pi_R(\CA_e( \CK))' \cap N = \C I $. So $z \in 1 \otimes P$ and it satisfies 
$z\gamma(a) = a z$, for all $a \in N$. In particular, if $g \in CK$
and $ a = \pi_R (c(g ) + c^*(g ))$ we have that $az = z\gamma(a) = -za = az$ since $N$ and $P$ commute. However $a$ is unitary, hence
$z = -z$, that is $z = 0$.  So we have that $ N^\prime \cap M_R = P$.
\end{proof}

\begin{Theorem}\label{thm:CAR}
 Let $R \in \CB(\CH)$ be an operator such that $0\leq R \leq 1$,  $S_t^*RS_t = R$ for all $t\geq 0$, and $R$ and $1-R$ are
invertible. Furthermore, suppose that  $R S_t\CH \subseteq S_t\CH$
 and $\text{Tr}(R|_{S_t \CH } - R|^2_{S_t \CH} ) = \infty$.
 Then the CAR flow $\alpha^R$ has the following properties:
 \begin{enumerate}
\item  it is equimodular with respect to $\omega_R$. 

\item  it is not modularly extendable.

\item the relative commutant index $(c_\alpha(t))_{t\geq 0}$ satisfies 
$1< c_\alpha(t) \leq 2$ for all $t\geq 0$.

\item if in addition $R$ is diagonalizable and $\frac{1}{2} \not\in \sigma(R)$, then $\ind_c(\alpha)=0$, in other words the coupling
index of $\alpha$ is zero.
 \end{enumerate}
\end{Theorem}
\begin{proof} Let $R$ as in the statement of the theorem be fixed. We write $\alpha= \alpha^R$.

(1)   It is clear that $\alpha_t$ is $\omega_R$-preserving.
By Lemma~\ref{cond-exp-car} applied to the subspace $S_t\CH$, there exists a unique  normal $\omega_R$-preserving conditional
expectation from $M_R$ onto $\alpha_t(M_R)= \pi_R(\CA(S_t\CH))''$. Therefore, it follows from Theorem~\ref{equi} that $\alpha_t$ is equimodular with respect to $\omega_R$.

(2) By the previous item, the CAR flow $\alpha$ is equimodular. 
By Lemma~\ref{rel}, it follows that for $t >0$
\[
M_R \cap \alpha_t(M_R)^\prime = \pi_R (A_e( (S_t\CH)^\perp) )'',
\]
However, we have that $M_R \cap \alpha_t(M_R)^\prime$ 
and $\alpha_t(M_R)$ can not generate $M_R$ as von Neumann algebra, since
the subspace generated by the action of $M_R \cap \alpha_t(M_R)^\prime$ 
and $\alpha_t(M_R)$ on the vaccum $\Omega \otimes \Omega$ is orthogonal
to $\pi_R(c(f))\Omega \otimes \Omega$ for all $f \in (S_tH)^\perp$.
Hence by \cite[Corollary 3.7]{BISS}, the endomorphism $\alpha_t$ is not modularly
extendable.

(3)  Let $f_0 \in (S_t\CH)^\perp = L^2(0,t)\otimes \CK$, with $\|f_0\| = 1$ 
and consider 
  $u(f_0) = \pi_R(c(f_0)) + \pi_R(c(f_0)^*)$ which is a self adjoint unitary. Let $N_\alpha(t) = (\alpha_t(M_R)'\cap M_R) \vee \alpha_t(M_R)$
  and let $p_{\alpha}(t) \in B(\CH_R)$ be the orthogonal projection 
  onto the closure of the subspace 
  $N_{\alpha}(t) \Omega \otimes \Omega$. By Lemma~\ref{rel},
  we have that
  $N_{\alpha}(t) =\pi_R (\CA_e((S_t\CH)^\perp))'' \vee \alpha_t(M_R)$, 
  and by the first paragraph of the proof of Lemma~\ref{rel} we have that
  $N_\alpha(t)$ is a factor for every $t\geq 0$. Therefore, we have
  that the relative commutant index set is $\CI_\alpha = [0,\infty)$.
  By straightforward however elaborate computations involving the 
  explicit form of elememts in the range of $p_\alpha(t)$, one can
  check that
  \[
   u(f_0) p_{\alpha}(t) u(f_0) + p_{\alpha}(t) = 1.     
  \]
Since $\alpha$ is equimodular, by Lemma~\ref{C-cond} there exists a
 unique $\omega_R$-preserving normal conditional expectation $E_t$ of $M_R$ onto  $N_\alpha(t)$. Let $E_t^{-1}$ denote the associated operator-valued weight from $N_\alpha(t)'$ to $M_R'$, and 
 let $J$ be the modular conjugation operator of $M_R$ with respect to 
 $\Omega \otimes \Omega$. 
 Let $F_t$ be the operator-valued weight from $M_R \vee \{p_\alpha(t)\}''$ to
 $M_R$ formally defined by $F_t(\cdot) = J E_t^{-1}(J \cdot J ) J$.  Since $J p_{\alpha}(t) J =  p_{\alpha}(t)$ and 
 $E_t^{-1}( p_{\alpha}(t) ) = 1$, (see \cite{kosaki}), we have that  
 \begin{align*}
  E_t^{-1}( I ) = F_t(I) &=F_t (  u(f_0) p_{\alpha}(t) u(f_0) + p_{\alpha}(t) )\\ 
  &=  u(f_0) F_t ( p_{\alpha}(t) ) u(f_0) + F_t( p_{\alpha}(t))\\
  &=u(f_0)  u(f_0) + I = 2I. 
\end{align*}
Thus $\text{Ind }E_t = 2$. 
 As $N_{\alpha}(t) \neq M_R $, we have that $1 < [M_R,~ N_{\alpha}] \leq 2$, that is to say $1< c_\alpha(t) \leq 2$.
 
(4) With the additional assumptions on $R$, we find ourselves in
the framework of \cite[Section 3.1]{Bikram}. Namely, $R$ and $1-R$
are invertible, $R$ is diagonalizable, $\frac{1}{2} \not\in\sigma(R)$
and furthermore $R$ commutes with $S_t$ for all $t$. Indeed, since $S_tR\CH \subseteq S_tR$ we have  that $S_tS_t^* R = R S_tS_t^*$.   Moreover, since $S_t^* R S_t = R$ we have that 
\[ 
S_tR = S_tS_t^*RS_t = RS_tS_t^* S_t = RS_t.
\]

Let $H_\alpha =\{ (t, H_{\alpha}(t)),\, t \geq 0 \} $ be the 
coupling superproduct system associated to
for the CAR flow $\alpha$. By Proposition~\ref{prop:superprod-in-space},
the map $\rho: H_\alpha \to (0,\infty) \times \CF_-(\CH)\otimes \CF_-(\CH)$ given by $\rho(t, T)= T\Omega \otimes \Omega$ is injective,
fiberwise isometric and measurable, and hence the image can identified
as a superproduct system with $H_\alpha$, with product given by
$\rho((t, T) \cdot (s, S)) = (t+s, TS \Omega \otimes \Omega)$.  For simplicity, we will denote $H_\alpha^\rho(t)$ the $t$ fiber of $\rho(H_\alpha)$ inside the space $\CF_-(\CH)\otimes \CF_-(\CH)$.
By \cite[Theorem 3.21]{Bikram} we have that for $t >0 $,
\begin{align}\label{eq-Halpha}
 H_{\alpha}^\rho(t)
=\cspan \{ (f_1\wedge \cdots \wedge f_n ) \otimes (g_1 \wedge \cdots \wedge g_m) 
& : ~m,n \in \N, (-1)^{n} = (-1)^{m}  \\
& f_1, \cdots, f_n, g_1, \cdots, g_m \in L^2(0, t)\otimes\CK 
\}.
\end{align}

Since $\alpha$ is equimodular, its coupling superproduct system $H_\alpha$ has a canonical unit $U_t$ (see Remark~\ref{rem:canonical-unit}). It is straightforward to check
that $U_t = \Gamma(S_t) \otimes \Gamma(S_t)$ where $\Gamma(S_t)$
denotes the second quantization of $S_t$. Hence in $H_\alpha^\rho(t)$
we obtain the unit $U_t \Omega \otimes \Omega = \Omega \otimes \Omega$.

We employ the concepts and techniques surrounding addits of superproduct
systems as introduced in \cite[Section 4]{MS}. In our context, an addit of $H_\alpha^\rho$ is a one-parameter measurable  family $b_t \in H_\alpha^\rho(t)$ such that 
\[
b_s + U_s b_t = b_{s+t}
\]
and it is called a centered addit if $\langle b_t, \Omega \otimes\Omega \rangle = 0$
for all $t\geq 0$. In order to show that $\ind_c(\alpha)=0$, it suffices
to show that the only centered addit of $H_\alpha^\rho$ is the zero addit, which corresponds to the canonical unit by \cite[Theorem 5.11]{MS}. We will follow an approach similar
to \cite[Lemma 7.1]{MS}.

Let $b = \{b_t\}_{t \geq 0 }$ be a centered addit for $H_\alpha^\rho$. Since $\CF_-(\CH)\otimes \CF_-(\CH) = \sum_{m,n \geq 0} \CH^{\wedge^m} \otimes \CH^{\wedge^n}$  and we have a corresponding orthogonal decomposition
\[
b = \sum_{m,n \geq 0} b^{m,n}
\]
where $b^{m,n} \in \CH^{\wedge^m} \otimes \CH^{\wedge^n}$ for all 
$m, n$.  As $\CH^{\wedge^m} \otimes \CH^{\wedge^n}$ is invariant 
under $U_t$, we have that $b^{m,n}$ is an addit for every $m,n$.

 It is straightforward to check that there exists $\lambda \in \C$
 such that $b_t^{0,0}   = \lambda t(\Omega \otimes \Omega)$  for all $t$. Now notice that whenever $m+n \geq 2$, by eq. \eqref{eq-Halpha}
  we can identify $b_t^{m,n}$
 with a function with appropriate symmetries on the set $[0,t]^{m+n}$.
 And for any partition $\{0 = t_0 < t_1 < \dots < t_\ell = t \}$
 we can write
  \[
 b^{m,n}_t = \sum_{j=0}^{\ell-1} S_{t_j} b^{m,n}_{t_{j+1} - t_j}
 \]
 which corresponds to a function with support in $\prod_{j=0}^{\ell -1}
 [t_j, t_{j+1}]\times [t_j, t_{j+1}]$. Since the partition was
 arbitrary, we see that the support of $b_t^{m,n}$ has to be a null set, in other words $b_t^{m,n} = 0$  for all $t$ and $m+n \geq 2$.
Thus we have that
\[
 b= b^{0,0} + b^{1,0} + b^{0,1}, 
\]
where $ b^{1,0}_s  \in  \CF(L^2(0, s)\otimes\CK)\otimes \Omega$   and  
$b^{0,1}_s \in  \Omega \otimes \CF(L^2(0, s)\otimes\CK)$. But 
by eq. \eqref{eq-Halpha} we have that $H_\alpha^\rho(s)$ does not
contain vectors of the form $ f \otimes \Omega$ or $\Omega \otimes f$, for $0\neq f \in L^2(0, \infty)\otimes\CK)$. 
We conclude that $b_t = b^{0,0}_t =\lambda t(\Omega \otimes \Omega)$,
and since $b_t$ is centered, we have that $\lambda = 0$ and $b_t = 0$
for all $t$. Hence $\ind_c(\alpha)=0$.
\end{proof}

\begin{Remark}\label{type}
We note that if $T \in \CB(\CK)$ is an operator such that 
$0 \leq T \leq 1$ with $T$ and $1-T$ invertible, then the operator
$R = 1 \otimes T$ on $\CH = L^2(0,\infty)\otimes \CK$ 
satisfies all the conditions of Theorem~\ref{thm:CAR}. Furthermore,
by varying $T$, the resulting factor $M_R$ may be chosen be the hyperfinite factors of type II$_1$, II$_\infty$ or III$_\lambda$ for $\lambda \in (0,1)$ by \cite[Lemma 5.3]{PoSt}. In particular, we have examples of non-modularly extendable E$_0$-semigroups on those 
factors. For the case of II$_\infty$ hyperfinite factors, this result is new.
\end{Remark}

 \begin{Corollary}\label{c-2}
 Let $\CK$ be a Hilbert space of any dimension and let $R \in \CB(\CH)$ be an operator satisfying the following properties:
  $0\leq R \leq 1$, $R$ and $1-R$ are invertible, 
   $\Tr(R|_{S_t\CH} - R|^2_{S_t\CH} ) = \infty$ and moreover
  $S_t^*RS_t = R$ and $R S_t \CH \subseteq \CH$ for all $t\geq 0$.
  Let $\alpha$ be the corresponding CAR flow on $M_R$.  
  Then for $k \neq \ell \in \N$,  we have that  ${\alpha}^{\otimes k}$ and $ {\alpha}^{\otimes \ell}$  are not cocycle conjugate when considered as E$_0$-semigroups on $M_R$. 
 \end{Corollary}
 \begin{proof}
 It is straightforward to check that $M_R \bar{\otimes} M_R \cong M_R$
 since both are hyperfinite factors and have the same Connes invariants. Thus,
 for every $k \neq  \ell \in \N$, we may consider $\alpha^{\otimes^k}$
 and $\alpha^{\otimes^\ell}$ as E$_0$-semigroups on the \emph{same} algebra $M_R$. The result now follows because by 
 Theorem~\ref{thm:CAR} and Proposition~\ref{prop:relative-commutant-index}
  both E$_0$-semigroups have different relative commutant index families.
\end{proof}
 
 \begin{Remark}
 It follows from Corollary~\ref{c-2} and Remark~\ref{type} that by varying $R\in \CB(\CH)$
 we obtain on every hyperfinite factor of types II$_1$, II$_\infty$
 and III$_\lambda$ for $\lambda \in (0,1)$ a countably infinite family
 of E$_0$-semigroups which are not modularly extendable and pairwise non-cocycle
 conjugate.
 \end{Remark}

\section*{Acknowledgments}
We would like to thank Prof. R. Srinivasan for his suggestions and comments.

\providecommand{\bysame}{\leavevmode\hbox to3em{\hrulefill}\thinspace}
\providecommand{\MR}{\relax\ifhmode\unskip\space\fi MR }
\providecommand{\MRhref}[2]{%
  \href{http://www.ams.org/mathscinet-getitem?mr=#1}{#2}
}
\providecommand{\href}[2]{#2}

\end{document}